\documentclass[10pt]{article}
\usepackage{amsmath}
\usepackage{amssymb}
\usepackage{amsthm}
\usepackage{mathrsfs}
\numberwithin{equation}{section}

\begin{document}

\title{Some estimates for $\theta$-type Calder\'on--Zygmund operators and linear commutators on certain weighted amalgam spaces}
\author{Hua Wang \footnote{E-mail address: wanghua@pku.edu.cn.}\\
\footnotesize{College of Mathematics and Econometrics, Hunan University, Changsha 410082, P. R. China}}
\date{}
\maketitle

\begin{abstract}
In this paper, we first introduce some new kinds of weighted amalgam spaces. Then we discuss the strong type and weak type estimates for a class of Calder\'on--Zygmund type operators $T_\theta$ in these new weighted spaces. Furthermore, the strong type estimate and endpoint estimate of linear commutators $[b,T_{\theta}]$ formed by $b$ and $T_{\theta}$ are established. Also we study related problems about two-weight, weak type inequalities for $T_{\theta}$ and $[b,T_{\theta}]$ in the weighted amalgam spaces and give some results.\\
MSC(2010): 42B20; 42B35; 46E30; 47B47\\
Keywords: $\theta$-type Calder\'on--Zygmund operators; commutators; weighted amalgam spaces; Muckenhoupt weights; Orlicz spaces.
\end{abstract}

\section{Introduction}

Calder\'on--Zygmund singular integral operators and their generalizations on the Euclidean space $\mathbb R^n$ have been extensively studied (see \cite{duoand,garcia,stein2,yabuta} for instance). In particular, Yabuta \cite{yabuta} introduced certain $\theta$-type Calder\'on--Zygmund operators to facilitate his study of certain classes of pseudo-differential operators. Following the terminology of Yabuta \cite{yabuta}, we introduce the so-called $\theta$-type Calder\'on--Zygmund operators.

\newtheorem{defn}{Definition}[section]

\begin{defn}
Let $\theta$ be a non-negative, non-decreasing function on $\mathbb R^+=(0,+\infty)$ with
\begin{equation}\label{theta1}
\int_0^1\frac{\theta(t)}{\,t\,}dt<\infty.
\end{equation}
A measurable function $K(\cdot,\cdot)$ on $\mathbb R^n\times\mathbb R^n\backslash\{(x,x):x\in\mathbb R^n\}$ is said to be a $\theta$-type kernel if it satisfies
\begin{align}
&(i)\quad \big|K(x,y)\big|\leq \frac{C}{|x-y|^{n}},\quad \mbox{for any }\, x\neq y;\\
&(ii)\quad \big|K(x,y)-K(z,y)\big|+\big|K(y,x)-K(y,z)\big|\leq \frac{C}{|x-y|^{n}}\cdot\theta\Big(\frac{|x-z|}{|x-y|}\Big), \\
&\qquad \mbox{for }\, |x-z|<|x-y|/2.\notag
\end{align}
\end{defn}

\begin{defn}
Let $T_\theta$ be a linear operator from $\mathscr S(\mathbb R^n)$ into its dual $\mathscr S'(\mathbb R^n)$. We say that $T_\theta$ is a $\theta$-type Calder\'on--Zygmund operator if

$(1)$ $T_\theta$ can be extended to be a bounded linear operator on $L^2(\mathbb R^n);$

$(2)$ There is a $\theta$-type kernel $K(x,y)$ such that
\begin{equation}
T_\theta f(x):=\int_{\mathbb R^n}K(x,y)f(y)\,dy
\end{equation}
for all $f\in C^\infty_0(\mathbb R^n)$ and for all $x\notin supp\,f$, where $C^\infty_0(\mathbb R^n)$ is the space consisting of all infinitely differentiable functions on $\mathbb R^n$ with compact supports.
\end{defn}
Note that the classical Calder\'on--Zygmund operator with standard kernel (see \cite{duoand,garcia}) is a special case of $\theta$-type operator $T_{\theta}$ when $\theta(t)=t^{\delta}$ with $0<\delta\leq1$.
\begin{defn}
Given a locally integrable function $b$ defined on $\mathbb R^n$, and given a $\theta$-type Calder\'on--Zygmund operator $T_{\theta}$, the linear commutator $[b,T_\theta]$ generated by $b$ and $T_{\theta}$ is defined for smooth, compactly supported functions $f$ as
\begin{equation}
\begin{split}
[b,T_\theta]f(x)&:=b(x)\cdot T_\theta f(x)-T_\theta(bf)(x)\\
&=\int_{\mathbb R^n}\big[b(x)-b(y)\big]K(x,y)f(y)\,dy.
\end{split}
\end{equation}
\end{defn}

\newtheorem{theorem}{Theorem}[section]

\newtheorem{corollary}{Corollary}[section]

\newtheorem{lemma}{Lemma}[section]

\newtheorem{rem}{Remark}[section]

We first give the following weighted result of $T_{\theta}$ obtained by Quek and Yang in \cite{quek}.

\begin{theorem}[\cite{quek}]\label{strongweak}
Suppose that $\theta$ is a non-negative, non-decreasing function on $\mathbb R^+=(0,+\infty)$ satisfying condition \eqref{theta1}. Let $1\leq p<\infty$ and $w\in A_p$. Then the $\theta$-type Calder\'on--Zygmund operator $T_{\theta}$ is bounded on $L^p_w(\mathbb R^n)$ for $p>1$, and bounded from $L^1_w(\mathbb R^n)$ into $WL^1_w(\mathbb R^n)$ for $p=1$.
\end{theorem}

Since linear commutator has a greater degree of singularity than the corresponding $\theta$-type Calder\'on--Zygmund operator, we need a slightly stronger condition (\ref{theta2}) given below. The following weighted endpoint estimate for commutator $[b,T_{\theta}]$ of the $\theta$-type Calder\'on--Zygmund operator was established in \cite{zhang2} under a stronger version of condition (\ref{theta2}) assumed on $\theta$, if $b\in BMO(\mathbb R^n)$ (for the unweighted case, see \cite{liu}). Let us now recall the definition of the space of $BMO(\mathbb R^n)$ (see \cite{duoand,john}). $BMO(\mathbb R^n)$ is the Banach function space modulo constants with the norm $\|\cdot\|_*$ defined by
\begin{equation*}
\|b\|_*:=\sup_{B}\frac{1}{|B|}\int_B|b(x)-b_B|\,dx<\infty,
\end{equation*}
where the supremum is taken over all balls $B$ in $\mathbb R^n$ and $b_B$ stands for the mean value of $b$ over $B$, that is,
\begin{equation*}
b_B:=\frac{1}{|B|}\int_B b(y)\,dy.
\end{equation*}

\begin{theorem}[\cite{zhang2}]\label{commutator}
Suppose that $\theta$ is a non-negative, non-decreasing function on $\mathbb R^+=(0,+\infty)$ satisfying \eqref{theta1} and
\begin{equation}\label{theta2}
\int_0^1\frac{\theta(t)\cdot|\log t|}{t}dt<\infty,
\end{equation}
let $w\in A_1$ and $b\in BMO(\mathbb R^n)$. Then for all $\lambda>0$, there is a constant $C>0$ independent of $f$ and $\lambda>0$ such that
\begin{equation*}
w\big(\big\{x\in\mathbb R^n:\big|[b,T_\theta](f)(x)\big|>\lambda\big\}\big)
\leq C\int_{\mathbb R^n}\Phi\left(\frac{|f(x)|}{\lambda}\right)\cdot w(x)\,dx,
\end{equation*}
where $\Phi(t)=t\cdot(1+\log^+t)$ and $\log^+t=\max\big\{\log t,0\big\}$.
\end{theorem}

We equip the $n$-dimensional Euclidean space $\mathbb R^n$ with the Euclidean norm $|\cdot|$ and the Lebesgue measure $dx$. For any $r>0$ and $y\in\mathbb R^n$, let $B(y,r)=\big\{x\in\mathbb R^n:|x-y|<r\big\}$ denote the open ball centered at $y$ with radius $r$, $B(y,r)^c$ denote its complement and $|B(y,r)|$ be the Lebesgue measure of the ball $B(y,r)$. We also use the notation $\chi_{B(y,r)}$ for the characteristic function of $B(y,r)$. Let $1\leq p,q,\alpha\leq\infty$. We define the amalgam space $(L^p,L^q)^{\alpha}(\mathbb R^n)$ of $L^p(\mathbb R^n)$ and $L^q(\mathbb R^n)$ as the set of all measurable functions $f$ satisfying $f\in L^p_{loc}(\mathbb R^n)$ and $\big\|f\big\|_{(L^p,L^q)^{\alpha}(\mathbb R^n)}<\infty$, where
\begin{equation*}
\begin{split}
\big\|f\big\|_{(L^p,L^q)^{\alpha}(\mathbb R^n)}
:=&\sup_{r>0}\left\{\int_{\mathbb R^n}\Big[\big|B(y,r)\big|^{1/{\alpha}-1/p-1/q}\big\|f\cdot\chi_{B(y,r)}\big\|_{L^p(\mathbb R^n)}\Big]^qdy\right\}^{1/q}\\
=&\sup_{r>0}\Big\|\big|B(y,r)\big|^{1/{\alpha}-1/p-1/q}\big\|f\cdot\chi_{B(y,r)}\big\|_{L^p(\mathbb R^n)}\Big\|_{L^q(\mathbb R^n)},
\end{split}
\end{equation*}
with the usual modification when $p=\infty$ or $q=\infty$. This amalgam space was originally introduced by Fofana in \cite{fofana}. As proved in \cite{fofana} the space $(L^p,L^q)^{\alpha}(\mathbb R^n)$ is nontrivial if and only if $p\leq\alpha\leq q$; thus in the remaining of the paper we will always assume that this condition $p\leq\alpha\leq q$ is fulfilled. Note that
\begin{itemize}
  \item For $1\leq p\leq\alpha\leq q\leq\infty$, one can easily see that $(L^p,L^q)^{\alpha}(\mathbb R^n)\subseteq(L^p,L^q)(\mathbb R^n)$, where $(L^p,L^q)(\mathbb R^n)$ is the Wiener amalgam space defined by (see \cite{F,holland} for more information)
\begin{equation*}
(L^p,L^q)(\mathbb R^n):=\left\{f:\big\|f\big\|_{(L^p,L^q)(\mathbb R^n)}
=\left(\int_{\mathbb R^n}\Big[\big\|f\cdot\chi_{B(y,1)}\big\|_{L^p(\mathbb R^n)}\Big]^qdy\right)^{1/q}<\infty\right\};
\end{equation*}
  \item If $1\leq p<\alpha$ and $q=\infty$, then $(L^p,L^q)^{\alpha}(\mathbb R^n)$ is just the classical Morrey space $\mathcal L^{p,\kappa}(\mathbb R^n)$ defined by (with $\kappa=1-p/{\alpha}$, see \cite{morrey})
\begin{equation*}
\mathcal L^{p,\kappa}(\mathbb R^n):=\left\{f:\big\|f\big\|_{\mathcal L^{p,\kappa}(\mathbb R^n)}
=\sup_{y\in\mathbb R^n,r>0}\left(\frac{1}{|B(y,r)|^\kappa}\int_{B(y,r)}|f(x)|^p\,dx\right)^{1/p}<\infty\right\};
\end{equation*}
  \item If $p=\alpha$ and $q=\infty$, then $(L^p,L^q)^{\alpha}(\mathbb R^n)$ reduces to the usual Lebesgue space $L^{\alpha}(\mathbb R^n)$.
\end{itemize}

In \cite{feuto2} (see also \cite{feuto1,feuto3}), Feuto considered a weighted version of the amalgam space $(L^p,L^q)^{\alpha}(w)$. A weight is any positive measurable function $w$ which is locally integrable on $\mathbb R^n$. Let $1\leq p\leq\alpha\leq q\leq\infty$ and $w$ be a weight on $\mathbb R^n$. We denote by $(L^p,L^q)^{\alpha}(w)$ the weighted amalgam space, the space of all locally integrable functions $f$ satisfying $\big\|f\big\|_{(L^p,L^q)^{\alpha}(w)}<\infty$, where
\begin{equation}\label{A}
\begin{split}
\big\|f\big\|_{(L^p,L^q)^{\alpha}(w)}
:=&\sup_{r>0}\left\{\int_{\mathbb R^n}\Big[w(B(y,r))^{1/{\alpha}-1/p-1/q}\big\|f\cdot\chi_{B(y,r)}\big\|_{L^p_w}\Big]^qdy\right\}^{1/q}\\
=&\sup_{r>0}\Big\|w(B(y,r))^{1/{\alpha}-1/p-1/q}\big\|f\cdot\chi_{B(y,r)}\big\|_{L^p_w}\Big\|_{L^q(\mathbb R^n)},
\end{split}
\end{equation}
with the usual modification when $q=\infty$ and $w(B(y,r))=\int_{B(y,r)}w(x)\,dx$ is the weighted measure of $B(y,r)$. Then for $1\leq p\leq\alpha\leq q\leq\infty$, we know that $(L^p,L^q)^{\alpha}(w)$ becomes a Banach function space with respect to the norm $\|\cdot\|_{(L^p,L^q)^{\alpha}(w)}$. Furthermore, we denote by $(WL^p,L^q)^{\alpha}(w)$ the weighted weak amalgam space of all measurable functions $f$ for which (see \cite{feuto2})
\begin{equation}\label{WA}
\begin{split}
\big\|f\big\|_{(WL^p,L^q)^{\alpha}(w)}
:=&\sup_{r>0}\left\{\int_{\mathbb R^n}\Big[w(B(y,r))^{1/{\alpha}-1/p-1/q}\big\|f\cdot\chi_{B(y,r)}\big\|_{WL^p_w}\Big]^qdy\right\}^{1/q}\\
=&\sup_{r>0}\Big\|w(B(y,r))^{1/{\alpha}-1/p-1/q}\big\|f\cdot\chi_{B(y,r)}\big\|_{WL^p_w}\Big\|_{L^q(\mathbb R^n)}<\infty.
\end{split}
\end{equation}

Note that
\begin{itemize}
  \item If $1\leq p<\alpha$ and $q=\infty$, then $(L^p,L^q)^{\alpha}(w)$ is just the weighted Morrey space $\mathcal L^{p,\kappa}(w)$ defined by (with $\kappa=1-p/{\alpha}$, see \cite{komori})
\begin{equation*}
\begin{split}
&\mathcal L^{p,\kappa}(w)\\
:=&\left\{f :\big\|f\big\|_{\mathcal L^{p,\kappa}(w)}
=\sup_{y\in\mathbb R^n,r>0}\left(\frac{1}{w(B(y,r))^{\kappa}}\int_{B(y,r)}|f(x)|^pw(x)\,dx\right)^{1/p}<\infty\right\},
\end{split}
\end{equation*}
and $(WL^p,L^q)^{\alpha}(w)$ is just the weighted weak Morrey space $W\mathcal L^{p,\kappa}(w)$ defined by (with $\kappa=1-p/{\alpha}$)
\begin{equation*}
\begin{split}
&W\mathcal L^{p,\kappa}(w)\\
:=&\left\{f :\big\|f\big\|_{W\mathcal L^{p,\kappa}(w)}
=\sup_{y\in\mathbb R^n,r>0}\sup_{\lambda>0}\frac{1}{w(B(y,r))^{\kappa/p}}\lambda\cdot\Big[w\big(\big\{x\in B(y,r):|f(x)|>\lambda\big\}\big)\Big]^{1/p}
<\infty\right\};
\end{split}
\end{equation*}
  \item If $p=\alpha$ and $q=\infty$, then $(L^p,L^q)^{\alpha}(w)$ reduces to the weighted Lebesgue space $L^{\alpha}_w(\mathbb R^n)$, and $(WL^p,L^q)^{\alpha}(w)$ reduces to the weighted weak Lebesgue space $WL^{\alpha}_w(\mathbb R^n)$.
\end{itemize}
Recently, many works in classical harmonic analysis have been devoted to norm inequalities involving several integral operators in the setting of weighted amalgam spaces, see \cite{feuto4,feuto1,feuto2,feuto3,wei}. These results obtained are extensions of well-known analogues in the weighted Lebesgue spaces. The main purpose of this paper is twofold. We first define some new kinds of weighted amalgam spaces, and then we are going to prove that $\theta$-type Calder\'on--Zygmund operator and associated linear commutator which are known to be bounded in weighted Lebesgue spaces, are also bounded in these new weighted
spaces under appropriate conditions. In addition, we will study two-weight, weak type norm inequalities for $\theta$-type Calder\'on--Zygmund operator and associated commutator in the context of weighted amalgam spaces.

Throughout this paper $C$ will denote a positive constant whose value may change at each appearance. We also use $A\approx B$ to denote the equivalence of $A$ and $B$; that is, there exist two positive constants $C_1$, $C_2$ independent of $A$ and $B$ such that $C_1 A\leq B\leq C_2 A$.

\section{Statements of the main results}

\subsection{Notations and preliminaries}
A weight $w$ is said to belong to the Muckenhoupt's class $A_p$ for $1<p<\infty$, if there exists a constant $C>0$ such that
\begin{equation*}
\left(\frac1{|B|}\int_B w(x)\,dx\right)^{1/p}\left(\frac1{|B|}\int_B w(x)^{-p'/p}\,dx\right)^{1/{p'}}\leq C
\end{equation*}
for every ball $B\subset\mathbb R^n$, where $p'$ is the dual of $p$ such that $1/p+1/{p'}=1$. The class $A_1$ is defined replacing the above inequality by
\begin{equation*}
\frac1{|B|}\int_B w(x)\,dx\leq C\cdot\underset{x\in B}{\mbox{ess\,inf}}\;w(x)
\end{equation*}
for every ball $B\subset\mathbb R^n$. We also define $A_\infty=\bigcup_{1\leq p<\infty}A_p$. For any given ball $B\subset\mathbb R^n$ and $\lambda>0$, we write $\lambda B$ for the ball with the same center as $B$ and radius is $\lambda$ times that of $B$. It is well known that if $w\in A_p$ with $1\leq p<\infty$(or $w\in A_\infty$), then $w$ satisfies the doubling condition; that is, for any ball $B\subset\mathbb R^n$, there exists an absolute constant $C>0$ such that (see \cite{garcia})
\begin{equation}\label{weights}
w(2B)\leq C\,w(B).
\end{equation}
When $w$ satisfies this doubling condition \eqref{weights}, we denote $w\in\Delta_2$ for brevity. Moreover, if $w\in A_\infty$, then for any ball $B$ and any measurable subset $E$ of a ball $B$, there exists a number $\delta>0$ independent of $E$ and $B$ such that (see \cite{garcia})
\begin{equation}\label{compare}
\frac{w(E)}{w(B)}\le C\left(\frac{|E|}{|B|}\right)^\delta.
\end{equation}

Given a weight $w$ on $\mathbb R^n$, as usual, the weighted Lebesgue space $L^p_w(\mathbb R^n)$ for $1\leq p<\infty$ is defined as the set of all functions $f$ such that
\begin{equation*}
\big\|f\big\|_{L^p_w}:=\bigg(\int_{\mathbb R^n}\big|f(x)\big|^pw(x)\,dx\bigg)^{1/p}<\infty.
\end{equation*}
We also denote by $WL^p_w(\mathbb R^n)$($1\leq p<\infty$) the weighted weak Lebesgue space consisting of all measurable functions $f$ such that
\begin{equation*}
\big\|f\big\|_{WL^p_w}:=
\sup_{\lambda>0}\lambda\cdot\Big[w\big(\big\{x\in\mathbb R^n:|f(x)|>\lambda\big\}\big)\Big]^{1/p}<\infty.
\end{equation*}

We next recall some basic definitions and facts about Orlicz spaces needed for the proof of the main results. For further information on the subject, one can see \cite{rao}. A function $\mathcal A$ is called a Young function if it is continuous, nonnegative, convex and strictly increasing on $[0,+\infty)$ with $\mathcal A(0)=0$ and $\mathcal A(t)\to +\infty$ as $t\to +\infty$. An important example of Young function is $\mathcal A(t)=t^p(1+\log^+t)^p$ with some $1\leq p<\infty$. Given a Young function $\mathcal A$, we define the $\mathcal A$-average of a function $f$ over a ball $B$ by means of the following Luxemburg norm:
\begin{equation*}
\big\|f\big\|_{\mathcal A,B}
:=\inf\left\{\lambda>0:\frac{1}{|B|}\int_B\mathcal A\left(\frac{|f(x)|}{\lambda}\right)dx\leq1\right\}.
\end{equation*}
When $\mathcal A(t)=t^p$, $1\leq p<\infty$, it is easy to see that
\begin{equation*}
\big\|f\big\|_{\mathcal A,B}=\left(\frac{1}{|B|}\int_B\big|f(x)\big|^p\,dx\right)^{1/p};
\end{equation*}
that is, the Luxemburg norm coincides with the normalized $L^p$ norm. Given a Young function $\mathcal A$, we use $\bar{\mathcal A}$ to denote the complementary Young function associated to $\mathcal A$. Then the following generalized H\"older's inequality holds for any given ball $B$:
\begin{equation*}
\frac{1}{|B|}\int_B\big|f(x)\cdot g(x)\big|\,dx\leq 2\big\|f\big\|_{\mathcal A,B}\big\|g\big\|_{\bar{\mathcal A},B}.
\end{equation*}
In particular, when $\mathcal A(t)=t\cdot(1+\log^+t)$, we know that its complementary Young function is $\bar{\mathcal A}(t)\approx\exp(t)-1$. In this situation, we denote
\begin{equation*}
\big\|f\big\|_{L\log L,B}=\big\|f\big\|_{\mathcal A,B}, \qquad
\big\|g\big\|_{\exp L,B}=\big\|g\big\|_{\bar{\mathcal A},B}.
\end{equation*}
So we have
\begin{equation}\label{holder}
\frac{1}{|B|}\int_B\big|f(x)\cdot g(x)\big|\,dx\leq 2\big\|f\big\|_{L\log L,B}\big\|g\big\|_{\exp L,B}.
\end{equation}

\subsection{Weighted amalgam spaces}

Let us begin with the definitions of the weighted amalgam spaces with Lebesgue measure in (\ref{A}) and (\ref{WA}) replaced by weighted measure.
\begin{defn}\label{amalgam}
Let $1\leq p\leq\alpha\leq q\leq\infty$, and let $w,\mu$ be two weights on $\mathbb R^n$. We denote by $(L^p,L^q)^{\alpha}(w;\mu)$ the weighted amalgam space, the space of all locally integrable functions $f$ with finite norm
\begin{equation*}
\begin{split}
\big\|f\big\|_{(L^p,L^q)^{\alpha}(w;\mu)}
:=&\sup_{r>0}\left\{\int_{\mathbb R^n}\Big[w(B(y,r))^{1/{\alpha}-1/p-1/q}\big\|f\cdot\chi_{B(y,r)}\big\|_{L^p_w}\Big]^q\mu(y)\,dy\right\}^{1/q}\\
=&\sup_{r>0}\Big\|w(B(y,r))^{1/{\alpha}-1/p-1/q}\big\|f\cdot\chi_{B(y,r)}\big\|_{L^p_w}\Big\|_{L^q_{\mu}}<\infty,
\end{split}
\end{equation*}
with the usual modification when $q=\infty$. Then we can see that the space $(L^p,L^q)^{\alpha}(w;\mu)$ equipped with the norm $\big\|\cdot\big\|_{(L^p,L^q)^{\alpha}(w;\mu)}$ is a Banach function space.
\end{defn}

\begin{defn}\label{Wamalgam}
Let $1\leq p\leq\alpha\leq q\leq\infty$, and let $w,\mu$ be two weights on $\mathbb R^n$. We denote by $(WL^p,L^q)^{\alpha}(w;\mu)$ the weighted weak amalgam space of all measurable functions $f$ for which
\begin{equation*}
\begin{split}
\big\|f\big\|_{(WL^p,L^q)^{\alpha}(w;\mu)}
:=&\sup_{r>0}\left\{\int_{\mathbb R^n}\Big[w(B(y,r))^{1/{\alpha}-1/p-1/q}\big\|f\cdot\chi_{B(y,r)}\big\|_{WL^p_w}\Big]^q\mu(y)\,dy\right\}^{1/q}\\
=&\sup_{r>0}\Big\|w(B(y,r))^{1/{\alpha}-1/p-1/q}\big\|f\cdot\chi_{B(y,r)}\big\|_{WL^p_w}\Big\|_{L^q_{\mu}}<\infty.
\end{split}
\end{equation*}
\end{defn}

We are going to prove that $\theta$-type Calder\'on--Zygmund operator which is known to be bounded on weighted Lebesgue spaces, is also bounded on these new weighted spaces for Muckenhoupt's weights. Our first two results in this paper can be formulated as follows.

\begin{theorem}\label{mainthm:1}
Let $1<p\leq\alpha<q\leq\infty$ and $w\in A_p$, $\mu\in\Delta_2$. Then the $\theta$-type Calder\'on--Zygmund operator $T_{\theta}$ is bounded on $(L^p,L^q)^{\alpha}(w;\mu)$.
\end{theorem}

\begin{theorem}\label{mainthm:2}
Let $p=1$, $1\leq\alpha<q\leq\infty$ and $w\in A_1$, $\mu\in\Delta_2$. Then the $\theta$-type Calder\'on--Zygmund operator $T_{\theta}$ is bounded from $(L^1,L^q)^{\alpha}(w;\mu)$ into $(WL^1,L^q)^{\alpha}(w;\mu)$.
\end{theorem}

Let $\theta$ be a non-negative, non-decreasing function on $\mathbb R^+=(0,+\infty)$ satisfying conditions $(\ref{theta1})$ and $(\ref{theta2})$, and let $[b,T_{\theta}]$ be the commutator formed by $T_{\theta}$ and BMO function $b$. For the strong type estimate of linear commutator $[b,T_{\theta}]$ on the weighted spaces $(L^p,L^q)^{\alpha}(w;\mu)$ with $1<p\leq\alpha<q$, we will prove

\begin{theorem}\label{mainthm:3}
Let $1<p\leq\alpha<q\leq\infty$ and $w\in A_p$, $\mu\in\Delta_2$. Assume that $\theta$ satisfies $(\ref{theta2})$ and $b\in BMO(\mathbb R^n)$, then the linear commutator $[b,T_{\theta}]$ is bounded on $(L^p,L^q)^{\alpha}(w;\mu)$.
\end{theorem}

To obtain endpoint estimate for the linear commutator $[b,T_{\theta}]$, we first need to define the weighted $\mathcal A$-average of a function $f$ over a ball $B$ by means of the weighted Luxemburg norm; that is, given a Young function $\mathcal A$ and $w\in A_\infty$, we define (see \cite{rao,zhang})
\begin{equation*}
\big\|f\big\|_{\mathcal A(w),B}:=\inf\left\{\sigma>0:\frac{1}{w(B)}
\int_B\mathcal A\left(\frac{|f(x)|}{\sigma}\right)\cdot w(x)\,dx\leq1\right\}.
\end{equation*}
When $\mathcal A(t)=t$, this norm is denoted by $\|\cdot\|_{L(w),B}$, when $\mathcal A(t)=t\cdot(1+\log^+t)$, this norm is also denoted by $\|\cdot\|_{L\log L(w),B}$. The complementary Young function of $t\cdot(1+\log^+t)$ is $\exp(t)-1$ with mean Luxemburg norm denoted by $\|\cdot\|_{\exp L(w),B}$. For $w\in A_\infty$ and for every ball $B$ in $\mathbb R^n$, we can also show the weighted version of \eqref{holder}. Namely, the following generalized H\"older's inequality in the weighted setting
\begin{equation}\label{Wholder}
\frac{1}{w(B)}\int_B|f(x)\cdot g(x)|w(x)\,dx\leq C\big\|f\big\|_{L\log L(w),B}\big\|g\big\|_{\exp L(w),B}
\end{equation}
is valid (see \cite{zhang} for instance). Furthermore, we now introduce new weighted spaces of $L\log L$ type as follows.

\begin{defn}
Let $p=1$, $1\leq\alpha\leq q\leq\infty$, and let $w,\mu$ be two weights on $\mathbb R^n$. We denote by $(L\log L,L^q)^{\alpha}(w;\mu)$ the weighted amalgam space of $L\log L$ type, the space of all locally integrable functions $f$ defined on $\mathbb R^n$ with finite norm
$\big\|f\big\|_{(L\log L,L^q)^{\alpha}(w;\mu)}$.
\begin{equation*}
(L\log L,L^q)^{\alpha}(w;\mu):=\left\{f:\big\|f\big\|_{(L\log L,L^q)^{\alpha}(w;\mu)}<\infty\right\},
\end{equation*}
where
\begin{equation*}
\begin{split}
\big\|f\big\|_{(L\log L,L^q)^{\alpha}(w;\mu)}
:=&\sup_{r>0}\left\{\int_{\mathbb R^n}\Big[w(B(y,r))^{1/{\alpha}-1/q}\big\|f\big\|_{L\log L(w),B(y,r)}\Big]^q\mu(y)\,dy\right\}^{1/q}\\
=&\sup_{r>0}\Big\|w(B(y,r))^{1/{\alpha}-1/q}\big\|f\big\|_{L\log L(w),B(y,r)}\Big\|_{L^q_{\mu}}.
\end{split}
\end{equation*}
\end{defn}

Observe that $t\leq t\cdot(1+\log^+t)$ for all $t>0$. Then for any ball $B(y,r)\subset\mathbb R^n$ and $w\in A_\infty$, we have $\big\|f\big\|_{L(w),B(y,r)}\leq \big\|f\big\|_{L\log L(w),B(y,r)}$ by definition, i.e., the inequality
\begin{equation}\label{main esti1}
\big\|f\big\|_{L(w),B(y,r)}=\frac{1}{w(B(y,r))}\int_{B(y,r)}|f(x)|\cdot w(x)\,dx\leq\big\|f\big\|_{L\log L(w),B(y,r)}
\end{equation}
holds for any ball $B(y,r)\subset\mathbb R^n$. Hence, for $1\leq\alpha\leq q\leq\infty$, we can further see the following inclusion:
\begin{equation*}
(L\log L,L^q)^{\alpha}(w;\mu)\subset (L^1,L^q)^{\alpha}(w;\mu).
\end{equation*}

For the endpoint case, we will prove the following weak type $L\log L$ estimate of the linear commutator $[b,T_{\theta}]$ in our weighted amalgam spaces.

\begin{theorem}\label{mainthm:4}
Let $p=1$, $1\leq\alpha<q\leq\infty$ and $w\in A_1$, $\mu\in\Delta_2$. Assume that $\theta$ satisfies $\eqref{theta2}$ and $b\in BMO(\mathbb R^n)$, then for any given $\lambda>0$ and any ball $B(y,r)\subset\mathbb R^n$ with $y\in\mathbb R^n$, $r>0$, there exists a constant $C>0$ independent of $f$, $B(y,r)$ and $\lambda>0$ such that
\begin{equation*}
\begin{split}
&\Big\|w(B(y,r))^{1/{\alpha}-1-1/q}\cdot w\big(\big\{x\in B(y,r):\big|[b,T_{\theta}](f)(x)\big|>\lambda\big\}\big)\Big\|_{L^q_{\mu}}\\
&\leq C\cdot\bigg\|\Phi\left(\frac{|f|}{\,\lambda\,}\right)\bigg\|_{(L\log L,L^q)^{\alpha}(w;\mu)},
\end{split}
\end{equation*}
where $\Phi(t)=t\cdot(1+\log^+t)$ and the norm $\|\cdot\|_{L^q_{\mu}}$ is taken with respect to the variable $y$, i.e.,
\begin{equation*}
\begin{split}
&\Big\|w(B(y,r))^{1/{\alpha}-1-1/q}\cdot w\big(\big\{x\in B(y,r):\big|[b,T_{\theta}](f)(x)\big|>\lambda\big\}\big)\Big\|_{L^q_{\mu}}\\
=&\left\{\int_{\mathbb R^n}\bigg[w(B(y,r))^{1/{\alpha}-1-1/q}\cdot w\big(\big\{x\in B(y,r):\big|[b,T_{\theta}](f)(x)\big|>\lambda\big\}\big)\bigg]^q\mu(y)\,dy\right\}^{1/q}.
\end{split}
\end{equation*}
\end{theorem}

\begin{rem}
From the above definitions and Theorem $\ref{mainthm:4}$, we can roughly say that the linear commutator $[b,T_{\theta}]$ is bounded from $(L\log L,L^q)^{\alpha}(w;\mu)$ into $(WL^1,L^q)^{\alpha}(w;\mu)$ whenever $1\leq\alpha<q\leq\infty$, $w\in A_1$ and $\mu\in\Delta_2$.
\end{rem}

\section{Proofs of Theorems \ref{mainthm:1} and \ref{mainthm:2}}

\begin{proof}[Proof of Theorem $\ref{mainthm:1}$]
Let $1<p\leq\alpha<q\leq\infty$ and $f\in(L^p,L^q)^{\alpha}(w;\mu)$ with $w\in A_p$ and $\mu\in\Delta_2$. We fix $y\in\mathbb R^n$ and $r>0$, and set $B=B(y,r)$ for the ball centered at $y$ and of radius $r$, $2B=B(y,2r)$. We represent $f$ as
\begin{equation*}
f=f\cdot\chi_{2B}+f\cdot\chi_{(2B)^c}:=f_1+f_2;
\end{equation*}
by the linearity of the $\theta$-type Calder\'on--Zygmund operator $T_{\theta}$, we write
\begin{align}\label{I}
&w(B(y,r))^{1/{\alpha}-1/p-1/q}\big\|T_\theta(f)\cdot\chi_{B(y,r)}\big\|_{L^p_w}\notag\\
&=w(B(y,r))^{1/{\alpha}-1/p-1/q}\bigg(\int_{B(y,r)}\big|T_\theta(f)(x)\big|^pw(x)\,dx\bigg)^{1/p}\notag\\
&\leq w(B(y,r))^{1/{\alpha}-1/p-1/q}\bigg(\int_{B(y,r)}\big|T_\theta(f_1)(x)\big|^pw(x)\,dx\bigg)^{1/p}\notag\\
&+w(B(y,r))^{1/{\alpha}-1/p-1/q}\bigg(\int_{B(y,r)}\big|T_\theta(f_2)(x)\big|^pw(x)\,dx\bigg)^{1/p}\notag\\
&:=I_1(y,r)+I_2(y,r).
\end{align}
Below we will give the estimates of $I_1(y,r)$ and $I_2(y,r)$, respectively. By the weighted $L^p$ boundedness of $T_{\theta}$ (see Theorem \ref{strongweak}), we have
\begin{align}\label{I1}
I_1(y,r)&\leq w(B(y,r))^{1/{\alpha}-1/p-1/q}\big\|T_\theta(f_1)\big\|_{L^p_w}\notag\\
&\leq C\cdot w(B(y,r))^{1/{\alpha}-1/p-1/q}
\bigg(\int_{B(y,2r)}|f(x)|^p w(x)\,dx\bigg)^{1/p}\notag\\
&=C\cdot w(B(y,2r))^{1/{\alpha}-1/p-1/q}\big\|f\cdot\chi_{B(y,2r)}\big\|_{L^p_w}\notag\\
&\times \frac{w(B(y,r))^{1/{\alpha}-1/p-1/q}}{w(B(y,2r))^{1/{\alpha}-1/p-1/q}}.
\end{align}
Moreover, since $1/{\alpha}-1/p-1/q<0$ and $w\in A_p$ with $1<p<\infty$, then by the inequality (\ref{weights}), we obtain
\begin{equation}\label{doubling1}
\frac{w(B(y,r))^{1/{\alpha}-1/p-1/q}}{w(B(y,2r))^{1/{\alpha}-1/p-1/q}}\leq C.
\end{equation}
Substituting the above inequality \eqref{doubling1} into \eqref{I1}, we thus obtain
\begin{equation}\label{I1yr}
I_1(y,r)\leq C\cdot w(B(y,2r))^{1/{\alpha}-1/p-1/q}\big\|f\cdot\chi_{B(y,2r)}\big\|_{L^p_w}.
\end{equation}
As for the term $I_2(y,r)$, it is clear that when $x\in B(y,r)$ and $z\in(B(y,2r))^c$, we get $|x-z|\approx|y-z|$. We then decompose $\mathbb R^n$ into a geometrically increasing sequence of concentric balls, and deduce the following pointwise estimate:
\begin{align}\label{pointwise1}
\big|T_\theta(f_2)(x)\big|&\leq C\int_{\mathbb R^n}\frac{|f_2(z)|}{|x-z|^n}dz
\leq C\int_{B(y,2r)^c}\frac{|f(z)|}{|y-z|^n}dz\notag\\
&=C\sum_{j=1}^\infty\int_{B(y,2^{j+1}r)\backslash B(y,2^jr)}\frac{|f(z)|}{|y-z|^n}dz\notag\\
&\leq C\sum_{j=1}^\infty\frac{1}{|B(y,2^{j+1}r)|}\int_{B(y,2^{j+1}r)}|f(z)|\,dz.
\end{align}
From this estimate \eqref{pointwise1}, it follows that
\begin{equation*}
I_2(y,r)\leq C\cdot w(B(y,r))^{1/{\alpha}-1/q}
\sum_{j=1}^\infty\frac{1}{|B(y,2^{j+1}r)|}\int_{B(y,2^{j+1}r)}|f(z)|\,dz.
\end{equation*}
By using H\"older's inequality and $A_p$ condition on $w$, we get
\begin{equation*}
\begin{split}
&\frac{1}{|B(y,2^{j+1}r)|}\int_{B(y,2^{j+1}r)}|f(z)|\,dz\\
&\leq\frac{1}{|B(y,2^{j+1}r)|}\bigg(\int_{B(y,2^{j+1}r)}|f(z)|^pw(z)\,dz\bigg)^{1/p}
\left(\int_{B(y,2^{j+1}r)}w(z)^{-{p'}/p}\,dz\right)^{1/{p'}}\\
&\leq C\bigg(\int_{B(y,2^{j+1}r)}|f(z)|^pw(z)\,dz\bigg)^{1/p}\cdot w\big(B(y,2^{j+1}r)\big)^{-1/p}.
\end{split}
\end{equation*}
Hence,
\begin{equation}\label{I2yr}
\begin{split}
I_2(y,r)&\leq C\cdot w(B(y,r))^{1/{\alpha}-1/q}\\
&\times\sum_{j=1}^\infty\bigg(\int_{B(y,2^{j+1}r)}|f(z)|^pw(z)\,dz\bigg)^{1/p}\cdot w\big(B(y,2^{j+1}r)\big)^{-1/p}\\
&=C\sum_{j=1}^\infty w\big(B(y,2^{j+1}r)\big)^{1/{\alpha}-1/p-1/q}\big\|f\cdot\chi_{B(y,2^{j+1}r)}\big\|_{L^p_w}\\
&\times\frac{w(B(y,r))^{1/{\alpha}-1/q}}{w(B(y,2^{j+1}r))^{1/{\alpha}-1/q}}.
\end{split}
\end{equation}
Notice that if $w\in A_p$ for $1\leq p<\infty$, then $w\in A_\infty$. By using the inequality (\ref{compare}) with exponent $\delta>0$ and the fact that $\alpha<q$, we find that
\begin{align}\label{psi1}
\sum_{j=1}^\infty\frac{w(B(y,r))^{1/{\alpha}-1/q}}{w(B(y,2^{j+1}r))^{1/{\alpha}-1/q}}
&\leq C\sum_{j=1}^\infty\left(\frac{|B(y,r)|}{|B(y,2^{j+1}r)|}\right)^{\delta(1/{\alpha}-1/q)}\notag\\
&= C\sum_{j=1}^\infty\left(\frac{1}{2^{(j+1)n}}\right)^{\delta(1/{\alpha}-1/q)}\notag\\
&\leq C,
\end{align}
where the last series is convergent since $\delta(1/{\alpha}-1/q)>0$. Therefore by taking the $L^q_{\mu}$-norm of both sides of \eqref{I}(with respect to the variable $y$), and then using Minkowski's inequality, \eqref{I1yr}, \eqref{I2yr} and \eqref{psi1}, we have
\begin{equation*}
\begin{split}
&\Big\|w(B(y,r))^{1/{\alpha}-1/p-1/q}\big\|T_\theta(f)\cdot\chi_{B(y,r)}\big\|_{L^p_w}\Big\|_{L^q_{\mu}}\\
&\leq\big\|I_1(y,r)\big\|_{L^q_{\mu}}+\big\|I_2(y,r)\big\|_{L^q_{\mu}}\\
&\leq C\Big\|w(B(y,2r))^{1/{\alpha}-1/p-1/q}\big\|f\cdot\chi_{B(y,2r)}\big\|_{L^p_w}\Big\|_{L^q_{\mu}}\\
&+C\sum_{j=1}^\infty\Big\|w\big(B(y,2^{j+1}r)\big)^{1/{\alpha}-1/p-1/q}\big\|f\cdot\chi_{B(y,2^{j+1}r)}\big\|_{L^p_w}\Big\|_{L^q_{\mu}}
\times\frac{w(B(y,r))^{1/{\alpha}-1/q}}{w(B(y,2^{j+1}r))^{1/{\alpha}-1/q}}\\
&\leq C\big\|f\big\|_{(L^p,L^q)^{\alpha}(w;\mu)}+C\big\|f\big\|_{(L^p,L^q)^{\alpha}(w;\mu)}
\times\sum_{j=1}^\infty\frac{w(B(y,r))^{1/{\alpha}-1/q}}{w(B(y,2^{j+1}r))^{1/{\alpha}-1/q}}\\
&\leq C\big\|f\big\|_{(L^p,L^q)^{\alpha}(w;\mu)}.
\end{split}
\end{equation*}
Thus, by taking the supremum over all $r>0$, we complete the proof of Theorem \ref{mainthm:1}.
\end{proof}

\begin{proof}[Proof of Theorem $\ref{mainthm:2}$]
Let $p=1$, $1\leq\alpha<q\leq\infty$ and $f\in(L^1,L^q)^{\alpha}(w;\mu)$ with $w\in A_1$ and $\mu\in\Delta_2$. For an arbitrary ball $B=B(y,r)\subset\mathbb R^n$ with $y\in\mathbb R^n$ and $r>0$, we represent $f$ as
\begin{equation*}
f=f\cdot\chi_{2B}+f\cdot\chi_{(2B)^c}:=f_1+f_2;
\end{equation*}
then by the linearity of the $\theta$-type Calder\'on--Zygmund operator $T_{\theta}$, one can write
\begin{align}\label{Iprime}
&w(B(y,r))^{1/{\alpha}-1-1/q}\big\|T_{\theta}(f)\cdot\chi_{B(y,r)}\big\|_{WL^1_w}\notag\\
&\leq 2\cdot w(B(y,r))^{1/{\alpha}-1-1/q}\big\|T_{\theta}(f_1)\cdot\chi_{B(y,r)}\big\|_{WL^1_w}\notag\\
&+2\cdot w(B(y,r))^{1/{\alpha}-1-1/q}\big\|T_{\theta}(f_2)\cdot\chi_{B(y,r)}\big\|_{WL^1_w}\notag\\
&:=I'_1(y,r)+I'_2(y,r).
\end{align}
We first consider the term $I'_1(y,r)$. By the weighted weak $(1,1)$ boundedness of $T_{\theta}$ (see Theorem \ref{strongweak}), we have
\begin{align}\label{I1prime}
I'_1(y,r)&\leq 2\cdot w(B(y,r))^{1/{\alpha}-1-1/q}\big\|T_\theta(f_1)\big\|_{WL^1_w}\notag\\
&\leq C\cdot w(B(y,r))^{1/{\alpha}-1-1/q}
\bigg(\int_{B(y,2r)}|f(x)| w(x)\,dx\bigg)\notag\\
&=C\cdot w(B(y,2r))^{1/{\alpha}-1-1/q}\big\|f\cdot\chi_{B(y,2r)}\big\|_{L^1_w}\notag\\
&\times \frac{w(B(y,r))^{1/{\alpha}-1-1/q}}{w(B(y,2r))^{1/{\alpha}-1-1/q}}.
\end{align}
Moreover, since $1/{\alpha}-1-1/q<0$ and $w\in A_1$, then we apply inequality \eqref{weights} to obtain that
\begin{equation}\label{doubling2}
\frac{w(B(y,r))^{1/{\alpha}-1-1/q}}{w(B(y,2r))^{1/{\alpha}-1-1/q}}\leq C.
\end{equation}
Substituting the above inequality \eqref{doubling2} into \eqref{I1prime}, we thus obtain
\begin{equation}\label{WI1yr}
I'_1(y,r)\leq C\cdot w(B(y,2r))^{1/{\alpha}-1-1/q}\big\|f\cdot\chi_{B(y,2r)}\big\|_{L^1_w}.
\end{equation}
As for the second term $I'_2(y,r)$, it follows directly from Chebyshev's inequality and the pointwise estimate \eqref{pointwise1} that
\begin{equation*}
\begin{split}
I'_2(y,r)&\leq2\cdot w(B(y,r))^{1/{\alpha}-1-1/q}\int_{B(y,r)}\big|T_\theta(f_2)(x)\big|w(x)\,dx\\
&\leq C\cdot w(B(y,r))^{1/{\alpha}-1/q}
\sum_{j=1}^\infty\frac{1}{|B(y,2^{j+1}r)|}\int_{B(y,2^{j+1}r)}|f(z)|\,dz.
\end{split}
\end{equation*}
Another application of $A_1$ condition on $w$ gives that
\begin{equation*}
\begin{split}
&\frac{1}{|B(y,2^{j+1}r)|}\int_{B(y,2^{j+1}r)}|f(z)|\,dz\\
&\leq C\cdot\frac{1}{w(B(y,2^{j+1}r))}\cdot\underset{z\in B(y,2^{j+1}r)}{\mbox{ess\,inf}}\;w(z)\int_{B(y,2^{j+1}r)}|f(z)|\,dz\\
&\leq C\cdot\frac{1}{w(B(y,2^{j+1}r))}\bigg(\int_{B(y,2^{j+1}r)}|f(z)|w(z)\,dz\bigg).
\end{split}
\end{equation*}
Consequently,
\begin{equation}\label{WI2yr}
\begin{split}
I'_2(y,r)&\leq C\cdot w(B(y,r))^{1/{\alpha}-1/q}\\
&\times\sum_{j=1}^\infty\bigg(\int_{B(y,2^{j+1}r)}|f(z)|w(z)\,dz\bigg)\cdot w\big(B(y,2^{j+1}r)\big)^{-1}\\
&=C\sum_{j=1}^\infty w\big(B(y,2^{j+1}r)\big)^{1/{\alpha}-1-1/q}\big\|f\cdot\chi_{B(y,2^{j+1}r)}\big\|_{L^1_w}\\
&\times\frac{w(B(y,r))^{1/{\alpha}-1/q}}{w(B(y,2^{j+1}r))^{1/{\alpha}-1/q}}.
\end{split}
\end{equation}
Therefore by taking the $L^q_{\mu}$-norm of both sides of \eqref{Iprime}(with respect to the variable $y$), and then using Minkowski's inequality, \eqref{WI1yr}, \eqref{WI2yr}, we have
\begin{equation*}
\begin{split}
&\Big\|w(B(y,r))^{1/{\alpha}-1-1/q}\big\|T_\theta(f)\cdot\chi_{B(y,r)}\big\|_{WL^1_w}\Big\|_{L^q_{\mu}}\\
&\leq\big\|I'_1(y,r)\big\|_{L^q_{\mu}}+\big\|I'_2(y,r)\big\|_{L^q_{\mu}}\\
&\leq C\Big\|w(B(y,2r))^{1/{\alpha}-1-1/q}\big\|f\cdot\chi_{B(y,2r)}\big\|_{L^1_w}\Big\|_{L^q_{\mu}}\\
&+C\sum_{j=1}^\infty\Big\|w\big(B(y,2^{j+1}r)\big)^{1/{\alpha}-1-1/q}\big\|f\cdot\chi_{B(y,2^{j+1}r)}\big\|_{L^1_w}\Big\|_{L^q_{\mu}}
\times\frac{w(B(y,r))^{1/{\alpha}-1/q}}{w(B(y,2^{j+1}r))^{1/{\alpha}-1/q}}\\
&\leq C\big\|f\big\|_{(L^1,L^q)^{\alpha}(w;\mu)}+C\big\|f\big\|_{(L^1,L^q)^{\alpha}(w;\mu)}
\times\sum_{j=1}^\infty\frac{w(B(y,r))^{1/{\alpha}-1/q}}{w(B(y,2^{j+1}r))^{1/{\alpha}-1/q}}\\
&\leq C\big\|f\big\|_{(L^1,L^q)^{\alpha}(w;\mu)},
\end{split}
\end{equation*}
where in the last inequality we have used inequality \eqref{psi1}. Thus, by taking the supremum over all $r>0$, we finish the proof of Theorem \ref{mainthm:2}.
\end{proof}

\section{Proofs of Theorems \ref{mainthm:3} and \ref{mainthm:4}}

For the results involving commutators, we need the following properties of $BMO$ functions.
\begin{lemma}\label{BMO}
Let $b$ be a function in $BMO(\mathbb R^n)$. Then

$(i)$ For every ball $B$ in $\mathbb R^n$ and for all $j\in\mathbb Z^+$,
\begin{equation*}
\big|b_{2^{j+1}B}-b_B\big|\leq C\cdot(j+1)\|b\|_*.
\end{equation*}

$(ii)$ For every ball $B$ in $\mathbb R^n$ and for all $w\in A_p$ with $1\leq p<\infty$,
\begin{equation*}
\bigg(\int_B\big|b(x)-b_B\big|^pw(x)\,dx\bigg)^{1/p}\leq C\|b\|_*\cdot w(B)^{1/p}.
\end{equation*}
\end{lemma}
\begin{proof}
For the proof of $(i)$, we refer the reader to \cite{stein2}. For the proof of $(ii)$, we refer the reader to \cite{wang}.
\end{proof}

\begin{proof}[Proof of Theorem $\ref{mainthm:3}$]
Let $1<p\leq\alpha<q\leq\infty$ and $f\in(L^p,L^q)^{\alpha}(w;\mu)$ with $w\in A_p$ and $\mu\in\Delta_2$. For each fixed ball $B=B(y,r)\subset\mathbb R^n$ with $y\in\mathbb R^n$ and $r>0$, as before, we represent $f$ as $f=f_1+f_2$, where $f_1=f\cdot\chi_{2B}$, $2B=B(y,2r)\subset\mathbb R^n$. By the linearity of the commutator operator $[b,T_{\theta}]$, we write
\begin{align}\label{J}
&w(B(y,r))^{1/{\alpha}-1/p-1/q}\big\|[b,T_{\theta}](f)\cdot\chi_{B(y,r)}\big\|_{L^p_w}\notag\\
&=w(B(y,r))^{1/{\alpha}-1/p-1/q}\bigg(\int_{B(y,r)}\big|[b,T_{\theta}](f)(x)\big|^pw(x)\,dx\bigg)^{1/p}\notag\\
&\leq w(B(y,r))^{1/{\alpha}-1/p-1/q}\bigg(\int_{B(y,r)}\big|[b,T_{\theta}](f_1)(x)\big|^pw(x)\,dx\bigg)^{1/p}\notag\\
&+w(B(y,r))^{1/{\alpha}-1/p-1/q}\bigg(\int_{B(y,r)}\big|[b,T_{\theta}](f_2)(x)\big|^pw(x)\,dx\bigg)^{1/p}\notag\\
&:=J_1(y,r)+J_2(y,r).
\end{align}
Since $T_{\theta}$ is bounded on $L^p_w(\mathbb R^n)$ for $1<p<\infty$ and $w\in A_p$, according to Theorem \ref{strongweak}, then by the well-known boundedness criterion for commutators of linear operators, which was obtained by Alvarez et al. in \cite{alvarez}, we know that $[b,T_{\theta}]$ is also bounded on $L^p_w(\mathbb R^n)$ for all $1<p<\infty$ and $w\in A_p$, whenever $b\in BMO(\mathbb R^n)$. This fact together with inequality (\ref{doubling1}) implies that
\begin{align}\label{J1yr}
J_1(y,r)&\leq w(B(y,r))^{1/{\alpha}-1/p-1/q}\big\|[b,T_{\theta}](f_1)\big\|_{L^p_w}\notag\\
&\leq C\cdot w(B(y,r))^{1/{\alpha}-1/p-1/q}
\bigg(\int_{B(y,2r)}|f(x)|^p w(x)\,dx\bigg)^{1/p}\notag\\
&=C\cdot w(B(y,2r))^{1/{\alpha}-1/p-1/q}\big\|f\cdot\chi_{B(y,2r)}\big\|_{L^p_w}\notag\\
&\times \frac{w(B(y,r))^{1/{\alpha}-1/p-1/q}}{w(B(y,2r))^{1/{\alpha}-1/p-1/q}}\notag\\
&\leq C\cdot w(B(y,2r))^{1/{\alpha}-1/p-1/q}\big\|f\cdot\chi_{B(y,2r)}\big\|_{L^p_w}.
\end{align}
Let us now turn to the estimate of $J_2(y,r)$. By definition, for any $x\in B(y,r)$, we have
\begin{equation*}
\big|[b,T_\theta](f_2)(x)\big|\leq\big|b(x)-b_{B}\big|\cdot\big|T_\theta(f_2)(x)\big|
+\Big|T_\theta\big([b_{B}-b]f_2\big)(x)\Big|.
\end{equation*}
In the proof of Theorem \ref{mainthm:1}, we have already shown that (see \eqref{pointwise1})
\begin{equation*}
\big|T_\theta(f_2)(x)\big|\leq C\sum_{j=1}^\infty\frac{1}{|B(y,2^{j+1}r)|}\int_{B(y,2^{j+1}r)}|f(z)|\,dz.
\end{equation*}
Following the same arguments as in \eqref{pointwise1}, we can also prove that
\begin{align}\label{pointwise2}
\Big|T_\theta\big([b_{B}-b]f_2\big)(x)\Big|
&\leq C\int_{\mathbb R^n}\frac{|[b_B-b(z)]f_2(z)|}{|x-z|^n}dz\\
&\leq C\int_{B(y,2r)^c}\frac{|[b_B-b(z)]f(z)|}{|y-z|^n}dz\notag\\
&=C\sum_{j=1}^\infty\int_{B(y,2^{j+1}r)\backslash B(y,2^jr)}\frac{|b(z)-b_{B}|\cdot|f(z)|}{|y-z|^n}dz\notag\\
&\leq C\sum_{j=1}^\infty\frac{1}{|B(y,2^{j+1}r)|}\int_{B(y,2^{j+1}r)}\big|b(z)-b_{B}\big|\cdot|f(z)|\,dz\notag.
\end{align}
Hence, from the above two pointwise estimates for $\big|T_\theta(f_2)(x)\big|$ and $\big|T_\theta\big([b_{B}-b]f_2\big)(x)\big|$, it follows that
\begin{equation*}
\begin{split}
J_2(y,r)&\leq C\cdot w(B(y,r))^{1/{\alpha}-1/p-1/q}\bigg(\int_B\big|b(x)-b_B\big|^pw(x)\,dx\bigg)^{1/p}\\
&\times\bigg(\sum_{j=1}^\infty\frac{1}{|B(y,2^{j+1}r)|}\int_{B(y,2^{j+1}r)}|f(z)|\,dz\bigg)\\
&+C\cdot w(B(y,r))^{1/{\alpha}-1/q}
\sum_{j=1}^\infty\frac{1}{|B(y,2^{j+1}r)|}\int_{B(y,2^{j+1}r)}\big|b_{B(y,2^{j+1}r)}-b_{B(y,r)}\big|\cdot|f(z)|\,dz\\
&+C\cdot w(B(y,r))^{1/{\alpha}-1/q}
\sum_{j=1}^\infty\frac{1}{|B(y,2^{j+1}r)|}\int_{B(y,2^{j+1}r)}\big|b(z)-b_{B(y,2^{j+1}r)}\big|\cdot|f(z)|\,dz\\
&:=J_3(y,r)+J_4(y,r)+J_5(y,r).
\end{split}
\end{equation*}
Below we will give the estimates of $J_3(y,r)$, $J_4(y,r)$ and $J_5(y,r)$, respectively. Using $(ii)$ of Lemma \ref{BMO}, H\"older's inequality and the $A_p$ condition, we obtain
\begin{equation*}
\begin{split}
J_3(y,r)&\leq C\|b\|_*\cdot w(B(y,r))^{1/{\alpha}-1/q}
\times\sum_{j=1}^\infty\bigg(\frac{1}{|B(y,2^{j+1}r)|}\int_{B(y,2^{j+1}r)}|f(z)|\,dz\bigg)\\
&\leq C\|b\|_*\cdot w(B(y,r))^{1/{\alpha}-1/q}
\sum_{j=1}^\infty\frac{1}{|B(y,2^{j+1}r)|}
\bigg(\int_{B(y,2^{j+1}r)}|f(z)|^pw(z)\,dz\bigg)^{1/p}\\
&\times\left(\int_{B(y,2^{j+1}r)}w(z)^{-{p'}/p}\,dz\right)^{1/{p'}}\\
&\leq C\|b\|_*\cdot w(B(y,r))^{1/{\alpha}-1/q}\\
&\times\sum_{j=1}^\infty\bigg(\int_{B(y,2^{j+1}r)}|f(z)|^pw(z)\,dz\bigg)^{1/p}\cdot w\big(B(y,2^{j+1}r)\big)^{-1/p}.
\end{split}
\end{equation*}
Applying $(i)$ of Lemma \ref{BMO}, H\"older's inequality and the $A_p$ condition, we can deduce that
\begin{equation*}
\begin{split}
J_4(y,r)&\leq C\|b\|_*\cdot w(B(y,r))^{1/{\alpha}-1/q}
\times\sum_{j=1}^\infty\frac{(j+1)}{|B(y,2^{j+1}r)|}\int_{B(y,2^{j+1}r)}|f(z)|\,dz\\
&\leq C\|b\|_*\cdot w(B(y,r))^{1/{\alpha}-1/q}
\sum_{j=1}^\infty\frac{(j+1)}{|B(y,2^{j+1}r)|}
\bigg(\int_{B(y,2^{j+1}r)}|f(z)|^pw(z)\,dz\bigg)^{1/p}\\
&\times\left(\int_{B(y,2^{j+1}r)}w(z)^{-{p'}/p}\,dz\right)^{1/{p'}}\\
&\leq C\|b\|_*\cdot w(B(y,r))^{1/{\alpha}-1/q}\\
&\times\sum_{j=1}^\infty\big(j+1\big)\cdot\bigg(\int_{B(y,2^{j+1}r)}|f(z)|^pw(z)\,dz\bigg)^{1/p}\cdot w\big(B(y,2^{j+1}r)\big)^{-1/p}.
\end{split}
\end{equation*}
It remains to estimate the last term $J_5(y,r)$. An application of H\"older's inequality gives us that
\begin{equation*}
\begin{split}
J_5(y,r)&\leq C\cdot w(B(y,r))^{1/{\alpha}-1/q}\sum_{j=1}^\infty\frac{1}{|B(y,2^{j+1}r)|}
\bigg(\int_{B(y,2^{j+1}r)}|f(z)|^pw(z)\,dz\bigg)^{1/p}\\
&\times\left(\int_{B(y,2^{j+1}r)}\big|b(z)-b_{B(y,2^{j+1}r)}\big|^{p'}w(z)^{-{p'}/p}\,dz\right)^{1/{p'}}.
\end{split}
\end{equation*}
If we set $\nu(z)=w(z)^{-{p'}/p}$, then we have $\nu\in A_{p'}$ because $w\in A_p$(see \cite{duoand,garcia}). Thus, it follows from $(ii)$ of Lemma \ref{BMO} and the $A_p$ condition that
\begin{equation*}
\begin{split}
&\left(\int_{B(y,2^{j+1}r)}\big|b(z)-b_{B(y,2^{j+1}r)}\big|^{p'}\nu(z)\,dz\right)^{1/{p'}}\\
&\leq C\|b\|_*\cdot\nu\big(B(y,2^{j+1}r)\big)^{1/{p'}}\\
&=C\|b\|_*\cdot\left(\int_{B(y,2^{j+1}r)}w(z)^{-{p'}/p}\,dz\right)^{1/{p'}}\\
&\leq C\|b\|_*\cdot\frac{|B(y,2^{j+1}r)|}{w(B(y,2^{j+1}r))^{1/p}}.
\end{split}
\end{equation*}
Therefore,
\begin{equation*}
\begin{split}
J_5(y,r)&\leq C\|b\|_*\cdot w(B(y,r))^{1/{\alpha}-1/q}\\
&\times\sum_{j=1}^\infty\bigg(\int_{B(y,2^{j+1}r)}|f(z)|^pw(z)\,dz\bigg)^{1/p}\cdot w\big(B(y,2^{j+1}r)\big)^{-1/p}.
\end{split}
\end{equation*}
Summarizing the above discussions, we conclude that
\begin{align}\label{J2yr}
J_2(y,r)&\leq C\|b\|_*\cdot w(B(y,r))^{1/{\alpha}-1/q}\notag\\
&\times\sum_{j=1}^\infty\big(j+1\big)\cdot\bigg(\int_{B(y,2^{j+1}r)}|f(z)|^pw(z)\,dz\bigg)^{1/p}\cdot w\big(B(y,2^{j+1}r)\big)^{-1/p}\notag\\
&=C\sum_{j=1}^\infty w\big(B(y,2^{j+1}r)\big)^{1/{\alpha}-1/p-1/q}\big\|f\cdot\chi_{B(y,2^{j+1}r)}\big\|_{L^p_w}\notag\\
&\times\big(j+1\big)\cdot\frac{w(B(y,r))^{1/{\alpha}-1/q}}{w(B(y,2^{j+1}r))^{1/{\alpha}-1/q}}.
\end{align}
Notice that when $w\in A_p$ with $1\leq p<\infty$, we have $w\in A_\infty$. Then by using inequality (\ref{compare}) with exponent $\delta>0$ together with the fact that $\alpha<q$, we thus obtain
\begin{align}\label{psi3}
\sum_{j=1}^\infty\big(j+1\big)\cdot\frac{w(B(y,r))^{1/{\alpha}-1/q}}{w(B(y,2^{j+1}r))^{1/{\alpha}-1/q}}
&\leq C\sum_{j=1}^\infty\big(j+1\big)\cdot\left(\frac{|B(y,r)|}{|B(y,2^{j+1}r)|}\right)^{\delta(1/{\alpha}-1/q)}\notag\\
&= C\sum_{j=1}^\infty\big(j+1\big)\cdot\left(\frac{1}{2^{(j+1)n}}\right)^{\delta(1/{\alpha}-1/q)}\notag\\
&\leq C,
\end{align}
where the last series is convergent since the exponent $\delta(1/{\alpha}-1/q)$ is positive.
Therefore by taking the $L^q_{\mu}$-norm of both sides of \eqref{J}(with respect to the variable $y$), and then using Minkowski's inequality, \eqref{J1yr}, \eqref{J2yr} and \eqref{psi3}, we can get
\begin{equation*}
\begin{split}
&\Big\|w(B(y,r))^{1/{\alpha}-1/p-1/q}\big\|[b,T_{\theta}](f)\cdot\chi_{B(y,r)}\big\|_{L^p_w}\Big\|_{L^q_{\mu}}\\
&\leq\big\|J_1(y,r)\big\|_{L^q_{\mu}}+\big\|J_2(y,r)\big\|_{L^q_{\mu}}\\
&\leq C\Big\|w(B(y,2r))^{1/{\alpha}-1/p-1/q}\big\|f\cdot\chi_{B(y,2r)}\big\|_{L^p_w}\Big\|_{L^q_{\mu}}\\
&+C\sum_{j=1}^\infty\Big\|w\big(B(y,2^{j+1}r)\big)^{1/{\alpha}-1/p-1/q}\big\|f\cdot\chi_{B(y,2^{j+1}r)}\big\|_{L^p_w}\Big\|_{L^q_{\mu}}\\
&\times\big(j+1\big)\cdot\frac{w(B(y,r))^{1/{\alpha}-1/q}}{w(B(y,2^{j+1}r))^{1/{\alpha}-1/q}}\\
&\leq C\big\|f\big\|_{(L^p,L^q)^{\alpha}(w;\mu)}+C\big\|f\big\|_{(L^p,L^q)^{\alpha}(w;\mu)}
\times\sum_{j=1}^\infty\big(j+1\big)\cdot\frac{w(B(y,r))^{1/{\alpha}-1/q}}{w(B(y,2^{j+1}r))^{1/{\alpha}-1/q}}\\
&\leq C\big\|f\big\|_{(L^p,L^q)^{\alpha}(w;\mu)}.
\end{split}
\end{equation*}
Thus, by taking the supremum over all $r>0$, we conclude the proof of Theorem \ref{mainthm:3}.
\end{proof}

\begin{proof}[Proof of Theorem $\ref{mainthm:4}$]
For any fixed ball $B=B(y,r)$ in $\mathbb R^n$, as before, we represent $f$ as $f=f_1+f_2$, where $f_1=f\cdot\chi_{2B}$, $2B=B(y,2r)\subset\mathbb R^n$. Then for any given $\lambda>0$, by the linearity of the commutator operator $[b,T_{\theta}]$, one can write
\begin{align}\label{Jprime}
&w(B(y,r))^{1/{\alpha}-1-1/q}\cdot w\big(\big\{x\in B(y,r):\big|[b,T_{\theta}](f)(x)\big|>\lambda\big\}\big)\notag\\
\leq &w(B(y,r))^{1/{\alpha}-1-1/q}\cdot w\big(\big\{x\in B(y,r):\big|[b,T_\theta](f_1)(x)\big|>\lambda/2\big\}\big)\notag\\
&+w(B(y,r))^{1/{\alpha}-1-1/q}\cdot w\big(\big\{x\in B(y,r):\big|[b,T_\theta](f_2)(x)\big|>\lambda/2\big\}\big)\notag\\
:=&J'_1(y,r)+J'_2(y,r).
\end{align}
In view of Theorem \ref{commutator}, we get
\begin{equation*}
\begin{split}
J'_1(y,r)&\leq C\cdot w(B(y,r))^{1/{\alpha}-1-1/q}
\int_{\mathbb R^n}\Phi\left(\frac{|f_1(x)|}{\lambda}\right)\cdot w(x)\,dx\\
&=C\cdot\frac{w(B(y,r))^{1/{\alpha}-1-1/q}}{w(B(y,2r))^{1/{\alpha}-1-1/q}}
\cdot\frac{w(B(y,2r))^{1/{\alpha}-1/q}}{w(B(y,2r))}
\int_{B(y,2r)}\Phi\left(\frac{|f(x)|}{\lambda}\right)\cdot w(x)\,dx.
\end{split}
\end{equation*}
Moreover, since $w\in A_1$, by the previous estimates \eqref{doubling2} and \eqref{main esti1}, we have
\begin{align}\label{WJ1yr}
J'_1(y,r)&\leq C\cdot\frac{w(B(y,2r))^{1/{\alpha}-1/q}}{w(B(y,2r))}
\int_{B(y,2r)}\Phi\left(\frac{|f(x)|}{\lambda}\right)\cdot w(x)\,dx\notag\\
&\leq C\cdot w(B(y,2r))^{1/{\alpha}-1/q}\bigg\|\Phi\left(\frac{|f|}{\,\lambda\,}\right)\bigg\|_{L\log L(w),B(y,2r)}.
\end{align}
We now turn to deal with the term $J'_2(y,r)$. Recall that the following inequality
\begin{equation*}
\big|[b,T_\theta](f_2)(x)\big|\leq\big|b(x)-b_{B(y,r)}\big|\cdot\big|T_\theta(f_2)(x)\big|
+\Big|T_\theta\big([b_{B(y,r)}-b]f_2\big)(x)\Big|
\end{equation*}
is valid. So we can further decompose $J'_2(y,r)$ as
\begin{equation*}
\begin{split}
J'_2(y,r)\leq&w(B(y,r))^{1/{\alpha}-1-1/q}\cdot
w\big(\big\{x\in B(y,r):\big|b(x)-b_{B(y,r)}\big|\cdot\big|T_\theta(f_2)(x)\big|>\lambda/4\big\}\big)\\
&+w(B(y,r))^{1/{\alpha}-1-1/q}\cdot
w\Big(\Big\{x\in B(y,r):\Big|T_\theta\big([b_{B(y,r)}-b]f_2\big)(x)\Big|>\lambda/4\Big\}\Big)\\
:=&J'_3(y,r)+J'_4(y,r).
\end{split}
\end{equation*}
By using the previous pointwise estimate \eqref{pointwise1}, Chebyshev's inequality together with $(ii)$ of Lemma \ref{BMO}, we can deduce that
\begin{equation*}
\begin{split}
J'_3(y,r)&\leq w(B(y,r))^{1/{\alpha}-1-1/q}\cdot\frac{\,4\,}{\lambda}
\int_{B(y,r)}\big|b(x)-b_{B(y,r)}\big|\cdot\big|T_\theta(f_2)(x)\big|w(x)\,dx\\
&\leq C\cdot w(B(y,r))^{1/{\alpha}-1/q}
\sum_{j=1}^\infty\frac{1}{|B(y,2^{j+1}r)|}\int_{B(y,2^{j+1}r)}\frac{|f(z)|}{\lambda}\,dz\\
&\times\frac{1}{w(B(y,r))}\int_{B(y,r)}\big|b(x)-b_{B(y,r)}\big|w(x)\,dx\\
&\leq C\|b\|_*\sum_{j=1}^\infty\frac{1}{|B(y,2^{j+1}r)|}\int_{B(y,2^{j+1}r)}\frac{|f(z)|}{\lambda}\,dz
\times w(B(y,r))^{1/{\alpha}-1/q}.
\end{split}
\end{equation*}
Furthermore, note that $t\leq\Phi(t)=t\cdot(1+\log^+t)$ for any $t>0$. It then follows from the $A_1$ condition and the previous estimate \eqref{main esti1} that
\begin{equation*}
\begin{split}
J'_3(y,r)&\leq C\|b\|_*\sum_{j=1}^\infty\frac{1}{w(B(y,2^{j+1}r))}\int_{B(y,2^{j+1}r)}\frac{|f(z)|}{\lambda}\cdot w(z)\,dz
\times w(B(y,r))^{1/{\alpha}-1/q}\\
&\leq C\|b\|_*\sum_{j=1}^\infty\frac{1}{w(B(y,2^{j+1}r))}\int_{B(y,2^{j+1}r)}\Phi\left(\frac{|f(z)|}{\lambda}\right)\cdot w(z)\,dz
\times w(B(y,r))^{1/{\alpha}-1/q}\\
&\leq C\|b\|_*\sum_{j=1}^\infty\bigg\|\Phi\left(\frac{|f|}{\,\lambda\,}\right)\bigg\|_{L\log L(w),B(y,2^{j+1}r)}
\times w(B(y,r))^{1/{\alpha}-1/q}.
\end{split}
\end{equation*}
On the other hand,
applying the pointwise estimate \eqref{pointwise2} and Chebyshev's inequality, we get
\begin{equation*}
\begin{split}
J'_4(y,r)&\leq w(B(y,r))^{1/{\alpha}-1-1/q}\cdot\frac{\,4\,}{\lambda}
\int_{B(y,r)}\Big|T_{\theta}\big([b_{B(y,r)}-b]f_2\big)(x)\Big|w(x)\,dx\\
&\leq w(B(y,r))^{1/{\alpha}-1/q}\cdot\frac{\,C\,}{\lambda}
\sum_{j=1}^\infty\frac{1}{|B(y,2^{j+1}r)|}\int_{B(y,2^{j+1}r)}\big|b(z)-b_{B(y,r)}\big|\cdot|f(z)|\,dz\\
&\leq w(B(y,r))^{1/{\alpha}-1/q}\cdot\frac{\,C\,}{\lambda}
\sum_{j=1}^\infty\frac{1}{|B(y,2^{j+1}r)|}\int_{B(y,2^{j+1}r)}\big|b(z)-b_{B(y,2^{j+1}r)}\big|\cdot|f(z)|\,dz\\
&+w(B(y,r))^{1/{\alpha}-1/q}\cdot\frac{\,C\,}{\lambda}
\sum_{j=1}^\infty\frac{1}{|B(y,2^{j+1}r)|}\int_{B(y,2^{j+1}r)}\big|b_{B(y,2^{j+1}r)}-b_{B(y,r)}\big|\cdot|f(z)|\,dz\\
&:=J'_5(y,r)+J'_6(y,r).
\end{split}
\end{equation*}
For the term $J'_5(y,r)$, since $w\in A_1$, it then follows from the $A_1$ condition and the fact $t\leq \Phi(t)$ that
\begin{equation*}
\begin{split}
J'_5(y,r)&\leq w(B(y,r))^{1/{\alpha}-1/q}\cdot\frac{\,C\,}{\lambda} \sum_{j=1}^\infty\frac{1}{w(B(y,2^{j+1}r))}\int_{B(y,2^{j+1}r)}\big|b(z)-b_{B(y,2^{j+1}r)}\big|\cdot|f(z)|w(z)\,dz\\
&\leq C\cdot w(B(y,r))^{1/{\alpha}-1/q}
\sum_{j=1}^\infty\frac{1}{w(B(y,2^{j+1}r))}\int_{B(y,2^{j+1}r)}\big|b(z)-b_{B(y,2^{j+1}r)}\big|
\cdot\Phi\left(\frac{|f(z)|}{\lambda}\right)w(z)\,dz.
\end{split}
\end{equation*}
Furthermore, we use the generalized H\"older's inequality with weight \eqref{Wholder} to obtain
\begin{equation*}
\begin{split}
J'_5(y,r)&\leq C\cdot w(B(y,r))^{1/{\alpha}-1/q}
\sum_{j=1}^\infty\big\|b-b_{B(y,2^{j+1}r)}\big\|_{\exp L(w),B(y,2^{j+1}r)}
\bigg\|\Phi\left(\frac{|f|}{\,\lambda\,}\right)\bigg\|_{L\log L(w),B(y,2^{j+1}r)}\\
&\leq C\|b\|_*\sum_{j=1}^\infty\bigg\|\Phi\left(\frac{|f|}{\,\lambda\,}\right)\bigg\|_{L\log L(w),B(y,2^{j+1}r)}
\times w(B(y,r))^{1/{\alpha}-1/q}.
\end{split}
\end{equation*}
In the last inequality, we have used the well-known fact that (see \cite{zhang})
\begin{equation}\label{Jensen}
\big\|b-b_{B}\big\|_{\exp L(w),B}\leq C\|b\|_*,\qquad \mbox{for any ball }B\subset\mathbb R^n.
\end{equation}
It is equivalent to the inequality (here $c_0$ is a universal constant)
\begin{equation*}
\frac{1}{w(B)}\int_B\exp\bigg(\frac{|b(z)-b_B|}{c_0\|b\|_*}\bigg)w(z)\,dz\leq C,
\end{equation*}
which is just a corollary of the well-known John--Nirenberg's inequality (see \cite{john}) and the comparison property of $A_1$ weights. For the last term $J'_6(y,r)$ we proceed as follows. Using $(i)$ of Lemma \ref{BMO} together with the facts that $w\in A_1$ and $t\leq\Phi(t)=t\cdot(1+\log^+t)$, we can deduce that
\begin{equation*}
\begin{split}
J'_6(y,r)&\leq C\cdot w(B(y,r))^{1/{\alpha}-1/q}
\sum_{j=1}^\infty(j+1)\|b\|_*\cdot\frac{1}{|B(y,2^{j+1}r)|}\int_{B(y,2^{j+1}r)}\frac{|f(z)|}{\lambda}\,dz\\
&\leq C\cdot w(B(y,r))^{1/{\alpha}-1/q}
\sum_{j=1}^\infty(j+1)\|b\|_*\cdot\frac{1}{w(B(y,2^{j+1}r))}\int_{B(y,2^{j+1}r)}\frac{|f(z)|}{\lambda}\cdot w(z)\,dz\\
&\leq C\|b\|_*\cdot w(B(y,r))^{1/{\alpha}-1/q}
\sum_{j=1}^\infty\frac{(j+1)}{w(B(y,2^{j+1}r))}\int_{B(y,2^{j+1}r)}\Phi\left(\frac{|f(z)|}{\lambda}\right)\cdot w(z)\,dz\\
&\leq C\|b\|_*\sum_{j=1}^\infty\big(j+1\big)\cdot\bigg\|\Phi\left(\frac{|f|}{\,\lambda\,}\right)\bigg\|_{L\log L(w),B(y,2^{j+1}r)}
\times w(B(y,r))^{1/{\alpha}-1/q},
\end{split}
\end{equation*}
where in the last inequality we have used the estimate \eqref{main esti1}. Summarizing the above discussions, we conclude that
\begin{align}\label{WJ2yr}
J'_2(y,r)&\leq C\|b\|_*\cdot\sum_{j=1}^\infty\big(j+1\big)\cdot\bigg\|\Phi\left(\frac{|f|}{\,\lambda\,}\right)\bigg\|_{L\log L(w),B(y,2^{j+1}r)}
\times w(B(y,r))^{1/{\alpha}-1/q}\notag\\
&=C\sum_{j=1}^\infty w\big(B(y,2^{j+1}r)\big)^{1/{\alpha}-1/q}
\bigg\|\Phi\left(\frac{|f|}{\,\lambda\,}\right)\bigg\|_{L\log L(w),B(y,2^{j+1}r)}\notag\\
&\times\big(j+1\big)\cdot\frac{w(B(y,r))^{1/{\alpha}-1/q}}{w(B(y,2^{j+1}r))^{1/{\alpha}-1/q}}.
\end{align}
Therefore by taking the $L^q_{\mu}$-norm of both sides of \eqref{Jprime}(with respect to the variable $y$), and then using Minkowski's inequality, \eqref{WJ1yr}, \eqref{WJ2yr}, we have
\begin{equation*}
\begin{split}
&\Big\|w(B(y,r))^{1/{\alpha}-1-1/q}\cdot w\big(\big\{x\in B(y,r):\big|[b,T_{\theta}](f)(x)\big|>\lambda\big\}\big)\Big\|_{L^q_{\mu}}\\
&\leq\big\|J'_1(y,r)\big\|_{L^q_{\mu}}+\big\|J'_2(y,r)\big\|_{L^q_{\mu}}\\
&\leq C\bigg\|w(B(y,2r))^{1/{\alpha}-1/q}\bigg\|\Phi\left(\frac{|f|}{\,\lambda\,}\right)\bigg\|_{L\log L(w),B(y,2r)}\bigg\|_{L^q_{\mu}}\\
&+C\sum_{j=1}^\infty\bigg\|w\big(B(y,2^{j+1}r)\big)^{1/{\alpha}-1/q}
\bigg\|\Phi\left(\frac{|f|}{\,\lambda\,}\right)\bigg\|_{L\log L(w),B(y,2^{j+1}r)}\bigg\|_{L^q_{\mu}}\\
&\times\big(j+1\big)\cdot\frac{w(B(y,r))^{1/{\alpha}-1/q}}{w(B(y,2^{j+1}r))^{1/{\alpha}-1/q}}\\
\end{split}
\end{equation*}
\begin{equation*}
\begin{split}
&\leq C\bigg\|\Phi\left(\frac{|f|}{\,\lambda\,}\right)\bigg\|_{(L\log L,L^q)^{\alpha}(w;\mu)}\\
&+C\bigg\|\Phi\left(\frac{|f|}{\,\lambda\,}\right)\bigg\|_{(L\log L,L^q)^{\alpha}(w;\mu)}
\times\sum_{j=1}^\infty\big(j+1\big)\cdot\frac{w(B(y,r))^{1/{\alpha}-1/q}}{w(B(y,2^{j+1}r))^{1/{\alpha}-1/q}}\\
&\leq C\bigg\|\Phi\left(\frac{|f|}{\,\lambda\,}\right)\bigg\|_{(L\log L,L^q)^{\alpha}(w;\mu)},
\end{split}
\end{equation*}
where the last inequality follows from \eqref{psi3}. This completes the proof of Theorem \ref{mainthm:4}.
\end{proof}

\section{Some results on two-weight problems}
In the last section, we consider related problems about two-weight, weak type $(p,p)$ inequalities with $1<p<\infty$. Let $\mathcal T$ be the classical Calder¨®n--Zygmund operator with standard kernel, that is, $\mathcal T=T_{\theta}$ when $\theta(t)=t^{\delta}$ with $0<\delta\leq1$. It is well known that $\mathcal T$ is a bounded operator on $L^p_w(\mathbb R^n)$ for all $1<p<\infty$ and $w\in A_p$, and of course, $\mathcal T$ is a bounded operator from $L^p_w(\mathbb R^n)$ into $WL^p_w(\mathbb R^n)$. In the two-weight context, however, the $A_p$ condition is NOT sufficient for the weak-type $(p,p)$ inequality for $\mathcal T$. More precisely, given a pair of weights $(u,v)$ and $p$, $1<p<\infty$, the weak-type inequality
\begin{equation}\label{T}
u\big(\big\{x\in\mathbb R^n:\big|\mathcal Tf(x)\big|>\lambda\big\}\big)
\leq \frac{C}{\lambda^p}\int_{\mathbb R^n}\big|f(x)\big|^p v(x)\,dx
\end{equation}
does not hold if $(u,v)\in A_p$: there exists a positive constant $C$ such that for every cube $Q\subset\mathbb R^n$,
\begin{equation}\label{two}
\left(\frac1{|Q|}\int_Q u(x)\,dx\right)^{1/p}\left(\frac1{|Q|}\int_Q v(x)^{-p'/p}\,dx\right)^{1/{p'}}\leq C<\infty,
\end{equation}
one can see \cite{cruz1,muckenhoupt} for some counter-examples. Here all cubes are assumed to have their sides parallel to the coordinate axes, $Q(x_0,\ell)$ will denote the cube centered at $x_0$ and has side length $\ell$. In \cite{cruz1,cruz2}, Cruz-Uribe and P\'erez considered the problem of finding sufficient conditions on a pair of weights $(u,v)$ such that $\mathcal T$ satisfies the weak-type $(p,p)$ inequality \eqref{T} ($1<p<\infty$). They showed in \cite{cruz2} that if we strengthened the $A_p$ condition \eqref{two} by adding a ``power bump" to the left-hand term, then inequality \eqref{T} holds for all $f\in L^p_v(\mathbb R^n)$. More specifically, if there exists a number $r>1$ such that for every cube $Q$ in $\mathbb R^n$,
\begin{equation}\label{assump1.1}
\left(\frac{1}{|Q|}\int_Q u(x)^r\,dx\right)^{1/{(rp)}}\left(\frac{1}{|Q|}\int_Q v(x)^{-p'/p}\,dx\right)^{1/{p'}}\leq C<\infty,
\end{equation}
then the classical Calder¨®n--Zygmund operator $\mathcal T$ is bounded from $L^p_v(\mathbb R^n)$ into $WL^p_u(\mathbb R^n)$. Moreover, in \cite{cruz1}, the authors improved this result by replacing the ``power bump" in \eqref{assump1.1} by a smaller ``Orlicz bump". To be more precise, they introduced the following $A_p$-type condition in the scale of Orlicz spaces:
\begin{equation*}
\big\|u\big\|_{L(\log L)^{p-1+\delta},Q}^{1/{p}}\left(\frac{1}{|Q|}\int_Q v(x)^{-p'/p}\,dx\right)^{1/{p'}}\leq C<\infty,\qquad \delta>0,
\end{equation*}
where $\big\|u\big\|_{L(\log L)^{p-1+\delta},Q}$ is the mean Luxemburg norm of $u$ on cube $Q$ with Young function $\mathcal A(t)=t\cdot(1+\log^+t)^{p-1+\delta}$. It was shown that inequality \eqref{T} still holds under the $A_p$-type condition on $(u,v)$, and this result is sharp since it does not hold in general when $\delta=0$.

On the other hand, the following Sharp function estimate for $T_{\theta}$ was established in \cite{liu}: there exists some $\delta$, $0<\delta<1$, and a positive constant $C=C_{\delta}$ such that for any $f\in C^\infty_0(\mathbb R^n)$ and $x\in\mathbb R^n$,
\begin{equation}\label{MJ}
\big[M^{\sharp}(|T_{\theta}f|^{\delta})(x)\big]^{1/{\delta}}\leq C\cdot M f(x),
\end{equation}
where $M$ is the standard Hardy--Littlewood maximal operator and $M^{\sharp}$ is the well-known Sharp maximal operator defined as
\begin{equation*}
M^{\sharp}f(x):=\sup_{x\in Q}\frac{1}{|Q|}\int_Q\big|f(y)-f_Q\big|\,dy.
\end{equation*}
Here the supremum is taken over all the cubes containing $x$ and $f_Q$ denotes the mean value of $f$ over $Q$, namely, $f_Q=\frac{1}{|Q|}\int_Q f(x)\,dx$. It was pointed out in \cite{cruz2} (Remark 1.3) that by using this Sharp function estimate \eqref{MJ}, we can also show inequality \eqref{T} is true for more general operator $T_{\theta}$, under the condition \eqref{assump1.1} on $(u,v)$. Then we obtain a sufficient condition for $T_{\theta}$ to be weak $(p,p)$ with $1<p<\infty$.
\begin{theorem}\label{WT}
Let $1<p<\infty$. Given a pair of weights $(u,v)$, suppose that for some $r>1$ and for all cubes $Q$ in $\mathbb R^n$,
\begin{equation*}
\left(\frac{1}{|Q|}\int_Q u(x)^r\,dx\right)^{1/{(rp)}}\left(\frac{1}{|Q|}\int_Q v(x)^{-p'/p}\,dx\right)^{1/{p'}}\leq C<\infty.
\end{equation*}
Then the $\theta$-type Calder\'on--Zygmund operator $T_{\theta}$ is bounded from $L^p_v(\mathbb R^n)$ into $WL^p_u(\mathbb R^n)$.
\end{theorem}

We will extend Theorem \ref{WT} to the weighted amalgam spaces. In order to do so, we need to define weighted amalgam spaces with two weights.
\begin{defn}
Let $1\leq p\leq\alpha\leq q\leq\infty$, and let $u,v,\mu$ be three weights on $\mathbb R^n$. We denote by $(L^p,L^q)^{\alpha}(v,u;\mu)$ the weighted amalgam space with two weights, the space of all locally integrable functions $f$ with finite norm
\begin{equation*}
\begin{split}
\big\|f\big\|_{(L^p,L^q)^{\alpha}(v,u;\mu)}
:=&\sup_{\ell>0}\left\{\int_{\mathbb R^n}\Big[u(Q(y,\ell))^{1/{\alpha}-1/p-1/q}
\big\|f\cdot\chi_{Q(y,\ell)}\big\|_{L^p_v}\Big]^q\mu(y)\,dy\right\}^{1/q}\\
=&\sup_{\ell>0}\Big\|u(Q(y,\ell))^{1/{\alpha}-1/p-1/q}
\big\|f\cdot\chi_{Q(y,\ell)}\big\|_{L^p_v}\Big\|_{L^q_{\mu}}<\infty,
\end{split}
\end{equation*}
with the usual modification when $q=\infty$. Alternatively, we could define the above notions of this section and section 2 with balls instead of cubes. We can also see that the space $(L^p,L^q)^{\alpha}(v,u;\mu)$ equipped with the norm $\big\|\cdot\big\|_{(L^p,L^q)^{\alpha}(v,u;\mu)}$ is a Banach function space.
\end{defn}
Note that
\begin{itemize}
  \item If $u=v=w$, then $(L^p,L^q)^{\alpha}(v,u;\mu)$ is the space $(L^p,L^q)^{\alpha}(w;\mu)$ in Definition \ref{amalgam};
  \item  If $1\leq p<\alpha$ and $q=\infty$, then $(L^p,L^q)^{\alpha}(v,u;\mu)$ is just the weighted Morrey space with two weights $\mathcal L^{p,\kappa}(v,u)$ defined by (with $\kappa=1-p/{\alpha}$, see \cite{komori})
\begin{equation*}
\begin{split}
&\mathcal L^{p,\kappa}(v,u)\\
:=&\left\{f :\big\|f\big\|_{\mathcal L^{p,\kappa}(v,u)}
=\sup_{y\in\mathbb R^n,\ell>0}\left(\frac{1}{u(Q(y,\ell))^{\kappa}}\int_{Q(y,\ell)}|f(x)|^pv(x)\,dx\right)^{1/p}<\infty\right\}.
\end{split}
\end{equation*}
\end{itemize}

We are now ready to prove the following result.

\begin{theorem}\label{mainthm:5}
Let $1<p\leq\alpha<q\leq\infty$ and $\mu\in\Delta_2$. Given a pair of weights $(u,v)$, suppose that for some $r>1$ and for all cubes $Q$ in $\mathbb R^n$,
\begin{equation*}
\left(\frac{1}{|Q|}\int_Q u(x)^r\,dx\right)^{1/{(rp)}}\left(\frac{1}{|Q|}\int_Q v(x)^{-p'/p}\,dx\right)^{1/{p'}}\leq C<\infty.
\end{equation*}
If $u\in \Delta_2$, then the $\theta$-type Calder\'on--Zygmund operator $T_{\theta}$ is bounded from $(L^p,L^q)^{\alpha}(v,u;\mu)$ into $(WL^p,L^q)^{\alpha}(u;\mu)$.
\end{theorem}

\begin{proof}[Proof of Theorem $\ref{mainthm:5}$]
Let $1<p\leq\alpha<q\leq\infty$ and $f\in(L^p,L^q)^{\alpha}(v,u;\mu)$ with $u\in\Delta_2$ and $\mu\in\Delta_2$. For any cube $Q=Q(y,\ell)\subset\mathbb R^n$ with $y\in\mathbb R^n$ and $\lambda>0$, we will denote by $\lambda Q$ the cube concentric with $Q$ whose each edge is $\lambda$ times as long, that is, $\lambda Q=Q(y,\lambda\ell)$. Let
\begin{equation*}
f=f\cdot\chi_{2Q}+f\cdot\chi_{(2Q)^c}:=f_1+f_2,
\end{equation*}
where $\chi_{2Q}$ denotes the characteristic function of $2Q=Q(y,2\ell)$. Then for given $y\in\mathbb R^n$ and $\ell>0$, we write
\begin{align}\label{K}
&u(Q(y,\ell))^{1/{\alpha}-1/p-1/q}\big\|T_{\theta}(f)\cdot\chi_{Q(y,\ell)}\big\|_{WL^p_u}\notag\\
&\leq 2\cdot u(Q(y,\ell))^{1/{\alpha}-1/p-1/q}\big\|T_{\theta}(f_1)\cdot\chi_{Q(y,\ell)}\big\|_{WL^p_u}\notag\\
&+2\cdot u(Q(y,\ell))^{1/{\alpha}-1/p-1/q}\big\|T_{\theta}(f_2)\cdot\chi_{Q(y,\ell)}\big\|_{WL^p_u}\notag\\
&:=K_1(y,\ell)+K_2(y,\ell).
\end{align}
In view of Theorem \ref{WT}, we get
\begin{align}\label{K1}
K_1(y,\ell)&\leq 2\cdot u(Q(y,\ell))^{1/{\alpha}-1/p-1/q}\big\|T_\theta(f_1)\big\|_{WL^p_u}\notag\\
&\leq C\cdot u(Q(y,\ell))^{1/{\alpha}-1/p-1/q}
\bigg(\int_{Q(y,2\ell)}|f(x)|^pv(x)\,dx\bigg)^{1/p}\notag\\
&=C\cdot u(Q(y,2\ell))^{1/{\alpha}-1/p-1/q}\big\|f\cdot\chi_{Q(y,2\ell)}\big\|_{L^p_v}\notag\\
&\times \frac{u(Q(y,\ell))^{1/{\alpha}-1/p-1/q}}{u(Q(y,2\ell))^{1/{\alpha}-1/p-1/q}}.
\end{align}
Moreover, since $1/{\alpha}-1/p-1/q<0$ and $u\in \Delta_2$, then by the inequality \eqref{weights}(consider cube $Q$ instead of ball $B$), we obtain
\begin{equation}\label{doubling3}
\frac{u(Q(y,\ell))^{1/{\alpha}-1/p-1/q}}{u(Q(y,2\ell))^{1/{\alpha}-1/p-1/q}}\leq C.
\end{equation}
Substituting the above inequality \eqref{doubling3} into \eqref{K1}, we thus obtain
\begin{equation}\label{k1yr}
K_1(y,\ell)\leq C\cdot u(Q(y,2\ell))^{1/{\alpha}-1/p-1/q}\big\|f\cdot\chi_{Q(y,2\ell)}\big\|_{L^p_v}.
\end{equation}
As for the term $K_2(y,\ell)$, using the same methods and steps as we deal with $I_2(y,r)$ in Theorem \ref{mainthm:1}, we can also show that for any $x\in Q(y,\ell)$,
\begin{equation}\label{pointwise3}
\big|T_{\theta}(f_2)(x)\big|\leq C
\sum_{j=1}^\infty\frac{1}{|Q(y,2^{j+1}\ell)|}\int_{Q(y,2^{j+1}\ell)}|f(z)|\,dz.
\end{equation}
This pointwise estimate \eqref{pointwise3} together with Chebyshev's inequality yields
\begin{equation*}
\begin{split}
K_2(y,\ell)&\leq 2\cdot u(Q(y,\ell))^{1/{\alpha}-1/p-1/q}\left(\int_{Q(y,\ell)}\big|T_{\theta}(f_2)(x)\big|^pu(x)\,dx\right)^{1/p}\\
&\leq C\cdot u(Q(y,\ell))^{1/{\alpha}-1/q}
\sum_{j=1}^\infty\frac{1}{|Q(y,2^{j+1}\ell)|}\int_{Q(y,2^{j+1}\ell)}|f(z)|\,dz.
\end{split}
\end{equation*}
Moreover, an application of H\"older's inequality gives us that
\begin{equation*}
\begin{split}
K_2(y,\ell)&\leq C\cdot u(Q(y,\ell))^{1/{\alpha}-1/q}
\sum_{j=1}^\infty\frac{1}{|Q(y,2^{j+1}\ell)|}\left(\int_{Q(y,2^{j+1}\ell)}|f(z)|^pv(z)\,dz\right)^{1/p}\\
&\times\left(\int_{Q(y,2^{j+1}\ell)}v(z)^{-p'/p}\,dz\right)^{1/{p'}}\\
&=C\sum_{j=1}^\infty u(Q(y,2^{j+1}\ell))^{1/{\alpha}-1/p-1/q}\big\|f\cdot\chi_{Q(y,2^{j+1}\ell)}\big\|_{L^p_v}\\
&\times\frac{u(Q(y,\ell))^{1/{\alpha}-1/q}}{u(Q(y,2^{j+1}\ell))^{1/{\alpha}-1/q}}
\cdot\frac{u(Q(y,2^{j+1}\ell))^{1/p}}{|Q(y,2^{j+1}\ell)|}\left(\int_{Q(y,2^{j+1}\ell)}v(z)^{-p'/p}\,dz\right)^{1/{p'}}.
\end{split}
\end{equation*}
In addition, we apply H\"older's inequality with exponent $r>1$ to get
\begin{equation}\label{U}
u\big(Q(y,2^{j+1}\ell)\big)=\int_{Q(y,2^{j+1}\ell)}u(z)\,dz
\leq\big|Q(y,2^{j+1}\ell)\big|^{1/{r'}}\left(\int_{Q(y,2^{j+1}\ell)}u(z)^r\,dz\right)^{1/r}.
\end{equation}
Consequently,
\begin{equation}\label{k2yr}
\begin{split}
K_2(y,\ell)&\leq C\sum_{j=1}^\infty u(Q(y,2^{j+1}\ell))^{1/{\alpha}-1/p-1/q}\big\|f\cdot\chi_{Q(y,2^{j+1}\ell)}\big\|_{L^p_v}
\cdot\frac{u(Q(y,\ell))^{1/{\alpha}-1/q}}{u(Q(y,2^{j+1}\ell))^{1/{\alpha}-1/q}}\\
&\times\frac{|Q(y,2^{j+1}\ell)|^{1/{(r'p)}}}{|Q(y,2^{j+1}\ell)|}
\left(\int_{Q(y,2^{j+1}\ell)}u(z)^r\,dz\right)^{1/{(rp)}}
\left(\int_{Q(y,2^{j+1}\ell)}v(z)^{-p'/p}\,dz\right)^{1/{p'}}\\
&\leq C\sum_{j=1}^\infty u(Q(y,2^{j+1}\ell))^{1/{\alpha}-1/p-1/q}\big\|f\cdot\chi_{Q(y,2^{j+1}\ell)}\big\|_{L^p_v}
\cdot\frac{u(Q(y,\ell))^{1/{\alpha}-1/q}}{u(Q(y,2^{j+1}\ell))^{1/{\alpha}-1/q}}.
\end{split}
\end{equation}
The last inequality is obtained by the $A_p$-type condition \eqref{assump1.1} on $(u,v)$. Furthermore, by our additional hypothesis on $u:u\in \Delta_2$, we can easily check that there exists a reverse doubling constant $D=D(u)>1$ independent of $Q$ such that (see Lemma 4.1 in \cite{komori})
\begin{equation*}
u(2Q)\geq D\cdot u(Q), \quad \mbox{for any cube }\,Q\subset\mathbb R^n,
\end{equation*}
which implies that for any $j\in\mathbb Z^+$, $u(2^{j+1}Q)\geq D^{j+1}\cdot u(Q)$ by iteration. Hence,
\begin{align}\label{5}
\sum_{j=1}^\infty\frac{u(Q(y,\ell))^{1/{\alpha}-1/q}}{u(Q(y,2^{j+1}\ell))^{1/{\alpha}-1/q}}
&\leq \sum_{j=1}^\infty\left(\frac{u(Q(y,\ell))}{D^{j+1}\cdot u(Q(y,\ell))}\right)^{1/{\alpha}-1/q}\notag\\
&=\sum_{j=1}^\infty\left(\frac{1}{D^{j+1}}\right)^{1/{\alpha}-1/q}\notag\\
&\leq C,
\end{align}
where the last series is convergent since the reverse doubling constant $D>1$ and $1/{\alpha}-1/q>0$. Therefore by taking the $L^q_{\mu}$-norm of both sides of \eqref{K}(with respect to the variable $y$), and then using Minkowski's inequality, \eqref{k1yr}, \eqref{k2yr} and \eqref{5}, we have
\begin{equation*}
\begin{split}
&\Big\|u(Q(y,\ell))^{1/{\alpha}-1/p-1/q}\big\|T_{\theta}(f)\cdot\chi_{Q(y,\ell)}\big\|_{WL^p_u}\Big\|_{L^q_{\mu}}\\
&\leq\big\|K_1(y,\ell)\big\|_{L^q_{\mu}}+\big\|K_2(y,\ell)\big\|_{L^q_{\mu}}\\
&\leq C\Big\|u(Q(y,2\ell))^{1/{\alpha}-1/p-1/q}\big\|f\cdot\chi_{Q(y,2\ell)}\big\|_{L^p_v}\Big\|_{L^q_{\mu}}\\
&+C\sum_{j=1}^\infty\Big\| u(Q(y,2^{j+1}\ell))^{1/{\alpha}-1/p-1/q}\big\|f\cdot\chi_{Q(y,2^{j+1}\ell)}\big\|_{L^p_v}\Big\|_{L^q_{\mu}}
\times\frac{u(Q(y,\ell))^{1/{\alpha}-1/q}}{u(Q(y,2^{j+1}\ell))^{1/{\alpha}-1/q}}\\
\end{split}
\end{equation*}
\begin{equation*}
\begin{split}
&\leq C\big\|f\big\|_{(L^p,L^q)^{\alpha}(v,u;\mu)}+C\big\|f\big\|_{(L^p,L^q)^{\alpha}(v,u;\mu)}
\times\sum_{j=1}^\infty\frac{u(Q(y,\ell))^{1/{\alpha}-1/q}}{u(Q(y,2^{j+1}\ell))^{1/{\alpha}-1/q}}\\
&\leq C\big\|f\big\|_{(L^p,L^q)^{\alpha}(v,u;\mu)}.
\end{split}
\end{equation*}
Finally, by taking the supremum over all $\ell>0$, we finish the proof of Theorem \ref{mainthm:5}.
\end{proof}

Let $M$ denote the Hardy--Littlewood maximal operator and $M^{\sharp}$ denote the Sharp maximal operator. For $\delta>0$, we define
\begin{equation*}
M_{\delta}(f):=\big[M(|f|^{\delta})\big]^{1/{\delta}},\qquad  M^{\sharp}_{\delta}(f):=\big[M^{\sharp}(|f|^{\delta})\big]^{1/{\delta}}.
\end{equation*}
The maximal function associated to $\mathcal A(t)=t\cdot(1+\log^+t)$ is defined as
\begin{equation*}
M_{L\log L}f(x):=\sup_{x\in Q}\big\|f\big\|_{L\log L,Q},
\end{equation*}
where the supremum is taken over all the cubes containing $x$. Let $b\in BMO(\mathbb R^n)$ and $[b,T_{\theta}]$ be the commutator of the $\theta$-type Calder\'on--Zygmund operator. In \cite{liu}, it was proved that if $\theta$ satisfies condition $(\ref{theta2})$, then for $0<\delta<\varepsilon<1$, there exists a positive constant $C=C_{\delta,\varepsilon}$ such that for any $f\in C^\infty_0(\mathbb R^n)$ and $x\in\mathbb R^n$,
\begin{equation}\label{MJ2}
M^{\sharp}_\delta([b,T_{\theta}]f)(x)\leq C\|b\|_*\Big(M_{\varepsilon}(T_{\theta}f)(x)+M_{L\log L}f(x)\Big).
\end{equation}
Using this Sharp function estimate \eqref{MJ2} and following the basic idea in \cite{cruz2}, we can also establish the two-weight, weak-type norm inequality for $[b,T_{\theta}]$.
\begin{theorem}\label{WT2}
Let $1<p<\infty$ and $b\in BMO(\mathbb R^n)$. Given a pair of weights $(u,v)$, suppose that for some $r>1$ and for all cubes $Q$ in $\mathbb R^n$,
\begin{equation*}
\left(\frac{1}{|Q|}\int_Q u(x)^r\,dx\right)^{1/{(rp)}}\big\|v^{-1/p}\big\|_{\mathcal A,Q}\leq C<\infty,
\end{equation*}
where $\mathcal A(t)=t^{p'}(1+\log^+t)^{p'}$ is a Young function. If $\theta$ satisfies $(\ref{theta2})$, then the commutator operator $[b,T_{\theta}]$ is bounded from $L^p_v(\mathbb R^n)$ into $WL^p_u(\mathbb R^n)$.
\end{theorem}

We will extend Theorem \ref{WT2} to the weighted amalgam spaces.For this purpose, we need the following key lemma.
\begin{lemma}\label{three}
Given three Young functions $\mathcal A$, $\mathcal B$ and $\mathcal C$ such that for all $t>0$,
\begin{equation*}
\mathcal A^{-1}(t)\cdot\mathcal B^{-1}(t)\leq\mathcal C^{-1}(t),
\end{equation*}
where $\mathcal A^{-1}(t)$ is the inverse function of $\mathcal A(t)$. Then we have the following generalized H\"older's inequality due to O'Neil \cite{neil}: for any cube $Q\subset\mathbb R^n$ and all functions $f$ and $g$,
\begin{equation*}
\big\|f\cdot g\big\|_{\mathcal C,Q}\leq 2\big\|f\big\|_{\mathcal A,Q}\big\|g\big\|_{\mathcal B,Q}.
\end{equation*}
\end{lemma}

\begin{theorem}\label{mainthm:6}
Let $1<p\leq\alpha<q\leq\infty$, $\mu\in\Delta_2$ and $b\in BMO(\mathbb R^n)$. Given a pair of weights $(u,v)$, suppose that for some $r>1$ and for all cubes $Q$ in $\mathbb R^n$,
\begin{equation}\label{assump1.2}
\left(\frac{1}{|Q|}\int_Q u(x)^r\,dx\right)^{1/{(rp)}}\big\|v^{-1/p}\big\|_{\mathcal A,Q}\leq C<\infty,
\end{equation}
where $\mathcal A(t)=t^{p'}(1+\log^+t)^{p'}$. If $\theta$ satisfies $(\ref{theta2})$ and $u\in A_\infty$, then the commutator operator $[b,T_{\theta}]$ is bounded from $(L^p,L^q)^{\alpha}(v,u;\mu)$ into $(WL^p,L^q)^{\alpha}(u;\mu)$.
\end{theorem}

\begin{proof}[Proof of Theorem $\ref{mainthm:6}$]
Let $1<p\leq\alpha<q\leq\infty$ and $f\in(L^p,L^q)^{\alpha}(v,u;\mu)$ with $u\in A_\infty$ and $\mu\in\Delta_2$. For an arbitrary cube $Q=Q(y,\ell)$ in $\mathbb R^n$, as before, we set
\begin{equation*}
f=f_1+f_2,\qquad f_1=f\cdot\chi_{2Q},\quad  f_2=f\cdot\chi_{(2Q)^c}.
\end{equation*}
Then for given $y\in\mathbb R^n$ and $\ell>0$, we write
\begin{align}\label{Kprime}
&u(Q(y,\ell))^{1/{\alpha}-1/p-1/q}\big\|[b,T_{\theta}](f)\cdot\chi_{Q(y,\ell)}\big\|_{WL^p_u}\notag\\
&\leq 2\cdot u(Q(y,\ell))^{1/{\alpha}-1/p-1/q}\big\|[b,T_{\theta}](f_1)\cdot\chi_{Q(y,\ell)}\big\|_{WL^p_u}\notag\\
&+2\cdot u(Q(y,\ell))^{1/{\alpha}-1/p-1/q}\big\|[b,T_{\theta}](f_2)\cdot\chi_{Q(y,\ell)}\big\|_{WL^p_u}\notag\\
&:=K'_1(y,\ell)+K'_2(y,\ell).
\end{align}
Since $u\in A_\infty$, we know that $u\in\Delta_2$. From Theorem \ref{WT2} and inequality (\ref{doubling3}), it follows that
\begin{align}\label{K1prime}
K'_1(y,\ell)&\leq 2\cdot u(Q(y,\ell))^{1/{\alpha}-1/p-1/q}\big\|[b,T_{\theta}](f_1)\big\|_{WL^p_u}\notag\\
&\leq C\cdot u(Q(y,\ell))^{1/{\alpha}-1/p-1/q}
\bigg(\int_{Q(y,2\ell)}|f(x)|^pv(x)\,dx\bigg)^{1/p}\notag\\
&=C\cdot u(Q(y,2\ell))^{1/{\alpha}-1/p-1/q}\big\|f\cdot\chi_{Q(y,2\ell)}\big\|_{L^p_v}\notag\\
&\times \frac{u(Q(y,\ell))^{1/{\alpha}-1/p-1/q}}{u(Q(y,2\ell))^{1/{\alpha}-1/p-1/q}}\notag\\
&\leq C\cdot u(Q(y,2\ell))^{1/{\alpha}-1/p-1/q}\big\|f\cdot\chi_{Q(y,2\ell)}\big\|_{L^p_v}.
\end{align}
Next we estimate $K'_2(y,\ell)$. For any $x\in Q(y,\ell)$, from the definition of $[b,T_{\theta}]$, we can see that
\begin{equation*}
\begin{split}
\big|[b,T_{\theta}](f_2)(x)\big|
&\leq \big|b(x)-b_{Q(y,\ell)}\big|\cdot\big|T_\theta(f_2)(x)\big|
+\Big|T_\theta\big([b_{Q(y,\ell)}-b]f_2\big)(x)\Big|\\
&:=\xi(x)+\eta(x).
\end{split}
\end{equation*}
Thus we have
\begin{equation*}
\begin{split}
K'_2(y,\ell)\leq&4\cdot u(Q(y,\ell))^{1/{\alpha}-1/p-1/q}\big\|\xi(\cdot)\cdot\chi_{Q(y,\ell)}\big\|_{WL^p_u}\\
&+4\cdot u(Q(y,\ell))^{1/{\alpha}-1/p-1/q}\big\|\eta(\cdot)\cdot\chi_{Q(y,\ell)}\big\|_{WL^p_u}\\
:=&K'_3(y,\ell)+K'_4(y,\ell).
\end{split}
\end{equation*}
For the term $K'_3(y,\ell)$, it follows directly from Chebyshev's inequality and estimate \eqref{pointwise3} that
\begin{equation*}
\begin{split}
K'_3(y,\ell)&\leq4\cdot u(Q(y,\ell))^{1/{\alpha}-1/p-1/q}\left(\int_{Q(y,\ell)}\big|\xi(x)\big|^pu(x)\,dx\right)^{1/p}\\
&\leq C\cdot u(Q(y,\ell))^{1/{\alpha}-1/p-1/q}\left(\int_{Q(y,\ell)}\big|b(x)-b_{Q(y,\ell)}\big|^pu(x)\,dx\right)^{1/p}\\
&\times\sum_{j=1}^\infty\frac{1}{|Q(y,2^{j+1}\ell)|}\int_{Q(y,2^{j+1}\ell)}|f(z)|\,dz\\
&\leq C\cdot u(Q(y,\ell))^{1/{\alpha}-1/q}
\sum_{j=1}^\infty\frac{1}{|Q(y,2^{j+1}\ell)|}\int_{Q(y,2^{j+1}\ell)}|f(z)|\,dz,
\end{split}
\end{equation*}
where in the last inequality we have used the fact that Lemma \ref{BMO}$(ii)$ still holds with ball $B$ replaced by cube $Q$, when $u$ is an $A_{\infty}$ weight. Repeating the arguments in the proof of Theorem \ref{mainthm:5}, we can also show that
\begin{equation*}
\begin{split}
K'_3(y,\ell)&\leq C\sum_{j=1}^\infty u(Q(y,2^{j+1}\ell))^{1/{\alpha}-1/p-1/q}\big\|f\cdot\chi_{Q(y,2^{j+1}\ell)}\big\|_{L^p_v}
\cdot\frac{u(Q(y,\ell))^{1/{\alpha}-1/q}}{u(Q(y,2^{j+1}\ell))^{1/{\alpha}-1/q}}.
\end{split}
\end{equation*}
As for the term $K'_4(y,\ell)$,
using the same methods and steps as we deal with $J_2(y,r)$ in Theorem \ref{mainthm:3}, we can show the following pointwise estimate as well.
\begin{equation*}
\begin{split}
\eta(x)&=\Big|T_\theta\big([b_{Q(y,\ell)}-b]f_2\big)(x)\Big|\\
&\leq C\sum_{j=1}^\infty\frac{1}{|Q(y,2^{j+1}\ell)|}\int_{Q(y,2^{j+1}\ell)}\big|b(z)-b_{Q(y,\ell)}\big|\cdot|f(z)|\,dz.
\end{split}
\end{equation*}
This, together with Chebyshev's inequality yields
\begin{equation*}
\begin{split}
K'_4(y,\ell)&\leq4\cdot u(Q(y,\ell))^{1/{\alpha}-1/p-1/q}\left(\int_{Q(y,\ell)}\big|\eta(x)\big|^pu(x)\,dx\right)^{1/p}\\
&\leq C\cdot u(Q(y,\ell))^{1/{\alpha}-1/q}
\sum_{j=1}^\infty\frac{1}{|Q(y,2^{j+1}\ell)|}\int_{Q(y,2^{j+1}\ell)}\big|b(z)-b_{Q(y,\ell)}\big|\cdot|f(z)|\,dz\\
&\leq C\cdot u(Q(y,\ell))^{1/{\alpha}-1/q}
\sum_{j=1}^\infty\frac{1}{|Q(y,2^{j+1}\ell)|}\int_{Q(y,2^{j+1}\ell)}\big|b(z)-b_{Q(y,2^{j+1}\ell)}\big|\cdot|f(z)|\,dz\\
&+C\cdot u(Q(y,\ell))^{1/{\alpha}-1/q}
\sum_{j=1}^\infty\frac{1}{|Q(y,2^{j+1}\ell)|}\int_{Q(y,2^{j+1}\ell)}\big|b_{Q(y,2^{j+1}\ell)}-b_{Q(y,\ell)}\big|\cdot|f(z)|\,dz\\
&:=K'_5(y,\ell)+K'_6(y,\ell).
\end{split}
\end{equation*}
An application of H\"older's inequality leads to that
\begin{equation*}
\begin{split}
K'_5(y,\ell)&\leq C\cdot u(Q(y,\ell))^{1/{\alpha}-1/q}
\sum_{j=1}^\infty\frac{1}{|Q(y,2^{j+1}\ell)|}\left(\int_{Q(y,2^{j+1}\ell)}|f(z)|^pv(z)\,dz\right)^{1/p}\\
&\times\left(\int_{Q(y,2^{j+1}\ell)}\big|b(z)-b_{Q(y,2^{j+1}\ell)}\big|^{p'}v(z)^{-p'/p}\,dz\right)^{1/{p'}}\\
&\leq C\cdot u(Q(y,\ell))^{1/{\alpha}-1/q}\sum_{j=1}^\infty\frac{\big\|f\cdot\chi_{Q(y,2^{j+1}\ell)}\big\|_{L^p_v}}{|Q(y,2^{j+1}\ell)|}\\
&\times\big|Q(y,2^{j+1}\ell)\big|^{1/{p'}}\Big\|(b-b_{Q(y,2^{j+1}\ell)})\cdot v^{-1/p}\Big\|_{\mathcal C,Q(y,2^{j+1}\ell)},
\end{split}
\end{equation*}
where $\mathcal C(t)=t^{p'}$ is a Young function. For $1<p<\infty$, we know the inverse function of $\mathcal C(t)$ is $\mathcal C^{-1}(t)=t^{1/{p'}}$. Observe that
\begin{equation*}
\begin{split}
\mathcal C^{-1}(t)&=t^{1/{p'}}\\
&=\frac{t^{1/{p'}}}{1+\log^+ t}\times\big(1+\log^+t\big)\\
&=\mathcal A^{-1}(t)\cdot\mathcal B^{-1}(t),
\end{split}
\end{equation*}
where
\begin{equation*}
\mathcal A(t)\approx t^{p'}(1+\log^+t)^{p'},\qquad \mbox{and}\qquad \mathcal B(t)\approx \exp(t)-1.
\end{equation*}
Thus, by Lemma \ref{three} and the estimate \eqref{Jensen}(consider cube $Q$ instead of ball $B$ when $w\equiv1$), we have
\begin{equation*}
\begin{split}
\Big\|(b-b_{Q(y,2^{j+1}\ell)})\cdot v^{-1/p}\Big\|_{\mathcal C,Q(y,2^{j+1}\ell)}
&\leq C\Big\|b-b_{Q(y,2^{j+1}\ell)}\Big\|_{\mathcal B,Q(y,2^{j+1}\ell)}\cdot\Big\|v^{-1/p}\Big\|_{\mathcal A,Q(y,2^{j+1}\ell)}\\
&\leq C\|b\|_*\cdot\Big\|v^{-1/p}\Big\|_{\mathcal A,Q(y,2^{j+1}\ell)}.
\end{split}
\end{equation*}
Moreover, in view of \eqref{U}, we can deduce that
\begin{equation*}
\begin{split}
K'_5(y,\ell)&\leq C\|b\|_*\cdot u(Q(y,\ell))^{1/{\alpha}-1/q}
\sum_{j=1}^\infty\frac{\big\|f\cdot\chi_{Q(y,2^{j+1}\ell)}\big\|_{L^p_v}}{|Q(y,2^{j+1}\ell)|^{1/p}}
\cdot\Big\|v^{-1/p}\Big\|_{\mathcal A,Q(y,2^{j+1}\ell)}\\
&=C\|b\|_*\sum_{j=1}^\infty u(Q(y,2^{j+1}\ell))^{1/{\alpha}-1/p-1/q}\big\|f\cdot\chi_{Q(y,2^{j+1}\ell)}\big\|_{L^p_v}
\cdot\frac{u(Q(y,\ell))^{1/{\alpha}-1/q}}{u(Q(y,2^{j+1}\ell))^{1/{\alpha}-1/q}}\\
&\times\frac{u(Q(y,2^{j+1}\ell))^{1/p}}{|Q(y,2^{j+1}\ell)|^{1/p}}
\cdot\Big\|v^{-1/p}\Big\|_{\mathcal A,Q(y,2^{j+1}\ell)}\\
&\leq C\|b\|_*\sum_{j=1}^\infty u(Q(y,2^{j+1}\ell))^{1/{\alpha}-1/p-1/q}\big\|f\cdot\chi_{Q(y,2^{j+1}\ell)}\big\|_{L^p_v}
\cdot\frac{u(Q(y,\ell))^{1/{\alpha}-1/q}}{u(Q(y,2^{j+1}\ell))^{1/{\alpha}-1/q}}\\
&\times\frac{|Q(y,2^{j+1}\ell)|^{1/{(r'p)}}}{|Q(y,2^{j+1}\ell)|^{1/p}}
\left(\int_{Q(y,2^{j+1}\ell)}u(z)^r\,dz\right)^{1/{(rp)}}\cdot\Big\|v^{-1/p}\Big\|_{\mathcal A,Q(y,2^{j+1}\ell)}\\
\end{split}
\end{equation*}
\begin{equation*}
\begin{split}
&\leq C\|b\|_*\sum_{j=1}^\infty u(Q(y,2^{j+1}\ell))^{1/{\alpha}-1/p-1/q}\big\|f\cdot\chi_{Q(y,2^{j+1}\ell)}\big\|_{L^p_v}
\cdot\frac{u(Q(y,\ell))^{1/{\alpha}-1/q}}{u(Q(y,2^{j+1}\ell))^{1/{\alpha}-1/q}}.
\end{split}
\end{equation*}
The last inequality is obtained by the $A_p$-type condition \eqref{assump1.2} on $(u,v)$. It remains to estimate the last term $K'_6(y,\ell)$. Applying Lemma \ref{BMO}$(i)$(use $Q$ instead of $B$) and H\"older's inequality, we get
\begin{equation*}
\begin{split}
K'_6(y,\ell)&\leq C\cdot u(Q(y,\ell))^{1/{\alpha}-1/q}
\sum_{j=1}^\infty\frac{(j+1)\|b\|_*}{|Q(y,2^{j+1}\ell)|}\int_{Q(y,2^{j+1}\ell)}|f(z)|\,dz\\
&\leq C\cdot u(Q(y,\ell))^{1/{\alpha}-1/q}
\sum_{j=1}^\infty\frac{(j+1)\|b\|_*}{|Q(y,2^{j+1}\ell)|}\left(\int_{Q(y,2^{j+1}\ell)}|f(z)|^pv(z)\,dz\right)^{1/p}\\
&\times\left(\int_{Q(y,2^{j+1}\ell)}v(z)^{-p'/p}\,dz\right)^{1/{p'}}\\
&=C\|b\|_*\sum_{j=1}^\infty u(Q(y,2^{j+1}\ell))^{1/{\alpha}-1/p-1/q}\big\|f\cdot\chi_{Q(y,2^{j+1}\ell)}\big\|_{L^p_v}\\
&\times\big(j+1\big)\cdot\frac{u(Q(y,\ell))^{1/{\alpha}-1/q}}{u(Q(y,2^{j+1}\ell))^{1/{\alpha}-1/q}}
\cdot\frac{u(Q(y,2^{j+1}\ell))^{1/p}}{|Q(y,2^{j+1}\ell)|}\left(\int_{Q(y,2^{j+1}\ell)}v(z)^{-p'/p}\,dz\right)^{1/{p'}}.
\end{split}
\end{equation*}
Let $\mathcal C(t)$, $\mathcal A(t)$ be the same as before. Obviously, $\mathcal C(t)\leq\mathcal A(t)$ for all $t>0$, then for any cube $Q\subset\mathbb R^n$, we have $\big\|f\big\|_{\mathcal C,Q}\leq\big\|f\big\|_{\mathcal A,Q}$ by definition, which implies that condition \eqref{assump1.2} is stronger that condition \eqref{assump1.1}. This fact together with \eqref{U} yields
\begin{equation*}
\begin{split}
K'_6(y,\ell)&\leq C\|b\|_*\sum_{j=1}^\infty u(Q(y,2^{j+1}\ell))^{1/{\alpha}-1/p-1/q}\big\|f\cdot\chi_{Q(y,2^{j+1}\ell)}\big\|_{L^p_v}\\
&\times\big(j+1\big)\cdot\frac{u(Q(y,\ell))^{1/{\alpha}-1/q}}{u(Q(y,2^{j+1}\ell))^{1/{\alpha}-1/q}}\\
&\times\frac{|Q(y,2^{j+1}\ell)|^{1/{(r'p)}}}{|Q(y,2^{j+1}\ell)|}
\left(\int_{Q(y,2^{j+1}\ell)}u(z)^r\,dz\right)^{1/{(rp)}}
\left(\int_{Q(y,2^{j+1}\ell)}v(z)^{-p'/p}\,dz\right)^{1/{p'}}\\
&\leq C\|b\|_*\sum_{j=1}^\infty u(Q(y,2^{j+1}\ell))^{1/{\alpha}-1/p-1/q}\big\|f\cdot\chi_{Q(y,2^{j+1}\ell)}\big\|_{L^p_v}\\
&\times\big(j+1\big)\cdot\frac{u(Q(y,\ell))^{1/{\alpha}-1/q}}{u(Q(y,2^{j+1}\ell))^{1/{\alpha}-1/q}}.
\end{split}
\end{equation*}
Summing up all the above estimates, we get
\begin{equation}\label{K2prime}
\begin{split}
K'_2(y,\ell)&\leq C\sum_{j=1}^\infty u(Q(y,2^{j+1}\ell))^{1/{\alpha}-1/p-1/q}\big\|f\cdot\chi_{Q(y,2^{j+1}\ell)}\big\|_{L^p_v}\\
&\times\big(j+1\big)\cdot\frac{u(Q(y,\ell))^{1/{\alpha}-1/q}}{u(Q(y,2^{j+1}\ell))^{1/{\alpha}-1/q}}.
\end{split}
\end{equation}
Moreover, by our additional hypothesis on $u:u\in A_\infty$ and inequality \eqref{compare} with exponent $\delta^\ast>0$(use $Q$ instead of $B$), we find that
\begin{align}\label{6}
\sum_{j=1}^\infty\big(j+1\big)\cdot\frac{u(Q(y,\ell))^{1/{\alpha}-1/q}}{u(Q(y,2^{j+1}\ell))^{1/{\alpha}-1/q}}
&\leq C\sum_{j=1}^\infty(j+1)\cdot\left(\frac{|Q(y,\ell)|}{|Q(y,2^{j+1}\ell)|}\right)^{\delta^\ast(1/{\alpha}-1/q)}\notag\\
&=C\sum_{j=1}^\infty(j+1)\cdot\left(\frac{1}{2^{(j+1)n}}\right)^{\delta^\ast(1/{\alpha}-1/q)}\notag\\
&\leq C.
\end{align}
Notice that the exponent $\delta^\ast(1/{\alpha}-1/q)$ is positive because $\alpha<q$, which guarantees that the last series is convergent.
Thus by taking the $L^q_{\mu}$-norm of both sides of \eqref{Kprime}(with respect to the variable $y$), and then using Minkowski's inequality, \eqref{K1prime}, \eqref{K2prime} and \eqref{6}, we finally obtain
\begin{equation*}
\begin{split}
&\Big\|u(Q(y,\ell))^{1/{\alpha}-1/p-1/q}\big\|[b,T_{\theta}](f)\cdot\chi_{Q(y,\ell)}\big\|_{WL^p_u}\Big\|_{L^q_{\mu}}\\
&\leq\big\|K'_1(y,\ell)\big\|_{L^q_{\mu}}+\big\|K'_2(y,\ell)\big\|_{L^q_{\mu}}\\
&\leq C\Big\|u(Q(y,2\ell))^{1/{\alpha}-1/p-1/q}\big\|f\cdot\chi_{Q(y,2\ell)}\big\|_{L^p_v}\Big\|_{L^q_{\mu}}\\
&+C\sum_{j=1}^\infty\Big\| u(Q(y,2^{j+1}\ell))^{1/{\alpha}-1/p-1/q}\big\|f\cdot\chi_{Q(y,2^{j+1}\ell)}\big\|_{L^p_v}\Big\|_{L^q_{\mu}}\\
&\times\big(j+1\big)\cdot\frac{u(Q(y,\ell))^{1/{\alpha}-1/q}}{u(Q(y,2^{j+1}\ell))^{1/{\alpha}-1/q}}\\
&\leq C\big\|f\big\|_{(L^p,L^q)^{\alpha}(v,u;\mu)}+C\big\|f\big\|_{(L^p,L^q)^{\alpha}(v,u;\mu)}
\times\sum_{j=1}^\infty\big(j+1\big)\cdot\frac{u(Q(y,\ell))^{1/{\alpha}-1/q}}{u(Q(y,2^{j+1}\ell))^{1/{\alpha}-1/q}}\\
&\leq C\big\|f\big\|_{(L^p,L^q)^{\alpha}(v,u;\mu)}.
\end{split}
\end{equation*}
We therefore conclude the proof of Theorem \ref{mainthm:6} by taking the supremum over all $\ell>0$.
\end{proof}

\end{document}